\documentclass[10pt, a4paper, reqno, oneside]{amsart}

\usepackage{amsmath, amsfonts, amsthm, amssymb, mathtools, amscd, enumerate, multicol, scalefnt, relsize}
\usepackage[mathscr]{euscript}
\usepackage{mathbbol}
\usepackage[latin1]{inputenc}
\usepackage{graphicx}
\usepackage[all]{xy}
\usepackage{enumitem}
\usepackage{autobreak,lipsum}
\setlist[itemize]{noitemsep, topsep=1pt, leftmargin=20pt}
\usepackage{xfrac}
\usepackage{todonotes}

\usepackage{fullpage}

\usetikzlibrary{arrows.meta,positioning}
\usetikzlibrary{backgrounds}

\usepackage[hypertexnames=false,
backref=page,
    pdfpagemode=UseNone,
    breaklinks=true,
    extension=pdf,
    colorlinks=true,
    linkcolor=blue,
    citecolor=blue,
    urlcolor=blue,
]{hyperref}

\newcommand\bcdot{\ensuremath{
  \mathchoice
   {\mskip\thinmuskip\lower0.2ex\hbox{\scalebox{1.6}{$\cdot$}}\mskip\thinmuskip}}
   {\mskip\thinmuskip\lower0.2ex\hbox{\scalebox{1.6}{$\cdot$}}\mskip\thinmuskip}
   {\lower0.3ex\hbox{\scalebox{1.2}{$\cdot$}}}
   {\lower0.3ex\hbox{\scalebox{1.2}{$\cdot$}}}
}
\theoremstyle{plain}
\newtheorem{theo}{Theorem}[section]

\theoremstyle{definition}

\newtheorem{definition}[theo]{Definition}

\theoremstyle{plain}
\newtheorem{lemma}[theo]{Lemma}
\newtheorem{theorem}[theo]{Theorem}
\newtheorem{corollary}[theo]{Corollary}
\newtheorem{proposition}[theo]{Proposition}

\theoremstyle{definition}

\newtheorem{remark}[theo]{Remark}

\theoremstyle{plain}
\newtheorem{thmint}{Theorem}

\theoremstyle{definition}
\newtheorem*{definition*}{Definition}

\DeclareSymbolFontAlphabet{\mathbb}{AMSb}
\DeclareSymbolFontAlphabet{\mathbbl}{bbold}

\makeatletter
\@namedef{subjclassname@2020}{\textup{2020} Mathematics Subject Classification}
\makeatother

\allowdisplaybreaks[4]

\title[]{On the long time behavior of ancient \\ homogeneous Ricci flows}

\author{Anusha M.\ Krishnan}
\address[Anusha M.\ Krishnan]{Department of Mathematics, Indian Institute of Technology Bombay \\ Powai, 400076 Mumbai, India}
\email{anushamk@math.iitb.ac.in}

\author{Francesco Pediconi}
\address[Francesco Pediconi]{Dipartimento di Scienze Matematiche ``Giuseppe Luigi Lagrange'' \\ Politecnico di Torino, corso Duca degli Abruzzi 24, 10129 Torino, Italy}
\email{francesco.pediconi@polito.it}

\thanks{The second-named author is a member of GNSAGA of INdAM}

\begin{document}

\begin{abstract}
We prove a precompactness theorem for invariant metrics on compact homogeneous spaces without injectivity radius bounds, assuming uniform bounds on the diameter and on all derivatives of the curvature tensor. As a consequence, we prove that every ancient homogeneous Ricci flow on a compact manifold admits a blow-down sequence that converges to a gradient shrinking Ricci soliton.
\end{abstract}

\maketitle

\section{Introduction}

Hamilton's Ricci flow is a weakly parabolic PDE that deforms a Riemannian metric in the direction of its negative Ricci tensor \cite{Ham82}. Much like the heat equation, it exhibits strong regularizing effects on Riemannian metrics, and has therefore become a fundamental tool in classification-type problems in both geometry and topology \cite{Perel02, BW08, BrSc09, BamKl23}. For this reason, the long-time behavior of the flow has been extensively studied in the literature, with particular emphasis on immortal solutions, i.e., those defined for all $t>0$, and on {\it ancient solutions}, i.e., those defined for all $t <0$. \smallskip

Classical structural results address {\it non-collapsed solutions}, namely those for which the injectivity radius of the corresponding curvature-normalized metrics admits a uniform positive lower bound. This assumption, combined with an appropriate curvature growth condition, allows one to apply Hamilton's Compactness Theorem \cite[Theorem 1.2]{Ham95'} to obtain a limit Ricci flow via a parabolic blow-down sequence, which can then be identified using a suitable monotone functional. This approach has been used to show that every non-collapsed Type-III solution on a compact manifold with diameter $O(\sqrt{t})$ admits a negative Einstein metric as a blow-down limit (see \cite{Ham99} and \cite[Theorem 1.3]{FIN05}), and that every non-collapsed Type-I ancient solution admits a non-flat shrinking Ricci soliton as a blow-down limit (see \cite[Theorem 4.1]{CaoZh11} and \cite[Theorem 3.1]{Na10}). \smallskip

In the absence of injectivity radius estimates, Hamilton's Compactness Theorem has been generalized by introducing a notion of convergence in the category of Riemannian groupoids (see \cite[Theorem 5.12]{Lott07}). As an application, results on the long time behavior of immortal Ricci flows have been obtained, even without non-collapsing assumptions. By \cite[Theorem 1.2]{Lott10}, any Type-III solution on a compact $3$-manifold with diameter $O(\sqrt{t})$ admits a blow-down sequence that converges smoothly, in the groupoid sense, to a homogeneous expanding Ricci soliton. Moreover, by \cite[Theorem 1]{BL18}, in the immortal homogeneous case an analogous result holds without any diameter or dimension assumptions. Here, we recall that a Ricci flow is called {\it homogeneous} if it is homogeneous at any time. \smallskip

No analogous result is currently known for collapsed ancient solutions, even in the homogeneous case. Our main theorem is the following.

\begin{thmint} \label{thm:MAIN}
Every ancient homogeneous Ricci flow on a compact manifold admits a blow-down sequence that converges locally in the $\mathcal{C}^{\infty}$-topology to a gradient shrinking Ricci soliton.
\end{thmint}

In the above theorem, {\it convergence locally in the $\mathcal{C}^{\infty}$-topology} can be understood as follows. By the Rauch Comparison Theorem, every homogeneous metric with sectional curvature $|{\rm sec}|\leq 1$ can be pulled back via the exponential map on a single tangent ball of radius $\pi$. This construction yields an incomplete locally homogeneous manifold, with injectivity radius at the origin equal to $\pi$, called a geometric model. By \cite[Theorem A]{Ped22}, every Riemannian locally homogeneous space with $|{\rm sec}|\leq 1$ has a unique geometric model. Therefore, since the geometry of a Riemannian (locally) homogeneous space is encoded at a single point, smooth convergence of geometric models provides a natural framework for describing collapsing phenomena in the homogeneous setting (see, e.g., \cite{BL18,BLS19,Ped22}). Thus, we say that a sequence $(M^{(n)},g^{(n)})$ of Riemannian homogeneous spaces converges locally in the $\mathcal{C}^{\infty}$-topology to a (possibly incomplete) Riemannian locally homogeneous space $(M^{(\infty)},g^{(\infty)})$ if the corresponding sequence of geometric models converges smoothly to the geometric model of $(M^{(\infty)},g^{(\infty)})$. We remark that this notion of convergence is closely related to convergence in the sense of Riemannian groupoids described in \cite[Definition 5.8]{Lott07}. We refer to Section \ref{subsect:conv} for further details. \smallskip

By the previously mentioned results \cite[Theorem 4.1]{CaoZh11} and \cite[Theorem 3.1]{Na10}, it is natural to expect a locally homogeneous gradient shrinking Ricci soliton as a blow-down limit of an ancient homogeneous Ricci flow. By \cite[Theorem 1.1]{PetWy09}, homogeneous gradient Ricci solitons are rigid, i.e., up to covering, they are isometric to the product of an Einstein manifold and a flat factor. Although the proof of \cite[Theorem 1.1]{PetWy09} relies on completeness of the metric, with only minor modifications it extends to show that locally homogeneous gradient Ricci solitons are rigid as well, as communicated to us by B{\"o}hm (see Appendix \ref{sect:appendix}). In fact, the solitons obtained in Theorem \ref{thm:MAIN} are locally isometric to the product of an Einstein metric with positive scalar curvature and a flat factor.
\smallskip

The main ingredient to prove Theorem \ref{thm:MAIN} is the following precompactness theorem, that is second main result of the paper.

\begin{thmint} \label{thm:MAIN'}
Let $M = \mathsf{G}/\mathsf{H}$ be a compact, connected homogeneous space, $D >0$ a constant and $\{\Gamma_k\}$ a sequence of positive numbers. Then, up to passing to a finite cover of $M$, every sequence $\{g^{(n)}\}$ of $\mathsf{G}$-invariant metrics on $M$ satisfying
$$
\mathrm{diam}\big(M,g^{(n)}\big) \leq D \quad \text{ and } \quad \sum_{i=0}^k \big|\big(D^{g^{(n)}}\big)^i\mathrm{Rm}(g^{(n)})\big|_{g^{(n)}} \leq \Gamma_k \,\, \text{for every $k \geq 0$} \,\, ,
$$
for all $n \in \mathbb{N}$, admits a subsequence for which the following hold: \begin{itemize}
\item[i)] there exists a compact torus $\mathsf{T} \subset N_{\mathsf{G}}(\mathsf{H})/\mathsf{H}$ acting freely on $M$ such that the metrics $g^{(n)}$ converge, in the Gromov--Hausdorff topology, to a $\mathsf{G}$-invariant metric $\check{g}$ on the quotient $M/\mathsf{T}$;
\item[ii)] the manifolds $(M,g^{(n)})$ converge locally in the $\mathcal{C}^{\infty}$-topology to the product of $(M/\mathsf{T},\check{g})$ and a flat factor.
\end{itemize}
\end{thmint}

The proof of Theorem \ref{thm:MAIN'} relies on two main steps. In the first one, we observe that every sequence of $\mathsf{G}$-invariant metrics $g^{(n)}$ as above subconverges in the $\mathsf{G}$-equivariant Gromov-Hausdorff topology to a compact Alexandrov space, which is in fact smooth by homogeneity. Then, the $\mathsf{G}$-equivariant version of Fukaya's Fibration Theorem \cite[Theorem 12.1']{Fu90}, together with its strengthened form \cite[Theorem 3.1, Theorem 3.2]{NaTian18}, allows us to deduce that, up to passing to a finite cover of $M$, the Gromov-Hausdorff limit of $(M,g^{(n)})$ is, in fact, diffeomorphic to the quotient of $M$ by a torus $\mathsf{T}$. Here, $\mathsf{T}$ is a closed abelian subgroup of a compact complement of $\mathsf{H}$ inside its normalizer $\mathsf{N}_{\mathsf{G}}(\mathsf{H})$, acting on the right on $M$ by $\mathsf{G}$-equivariant diffeomorphisms. Furthermore, quantitative control on the canonical projection $M \to M/\mathsf{T}$ along the sequence is established. We refer to Theorem \ref{thm:equivGH1} for more details.

Subsequently, we use a symmetrization result for collapsed metrics from \cite{ChFuGro92} (see Theorem \ref{thm:equivGH2}), that allows us to perform a dimensional reduction with respect to the projection $M \to M/\mathsf{T}$. Together with the quantitative estimates from Theorem \ref{thm:equivGH1}, this allows us to characterize the limit of $(M,g^{(n)})$, up to passing to a subsequence, in the local $\mathcal{C}^{\infty}$-topology. We refer to Theorem \ref{thm:equivGH3} for more details. \smallskip

We now outline the strategy to prove Theorem \ref{thm:MAIN}. Let $g(t)$ be an ancient homogeneous Ricci flow on a compact manifold $M = \mathsf{G}/\mathsf{H}$, where $\mathsf{G}$ is a compact Lie group and $\mathsf{H} \subset \mathsf{G}$ is a closed subgroup. If $g(t)$ is non-collapsed, then the result is already known (see, e.g., \cite[Theorem 4.2]{BLS19}). Thus, we focus on the collapsed case. By \cite[Corollary 2]{BLS19}, the solution $g(t)$ is Type-I and, as a consequence, its diameter is $O(\sqrt{|t|})$ by \cite[Theorem 17.2]{Ham95}. Notice that in the immortal case, Type-III curvature growth does not by itself imply such a diameter bound, which is essential to ensure that the blow-down limits remain compact.

As a consequence of these curvature and diameter bounds, for every blow-down sequence
$$
g^{(n)}(t) \coloneqq \frac1{\tau^{(n)}}\,g\big(\tau^{(n)}t\big) \,\, , \quad \tau^{(n)} \to +\infty
$$
along $g(t)$, the metrics $g^{(n)}(-1)$ have uniformly bounded diameter and uniformly bounded curvature tensor, together with all its covariant derivatives of arbitrary order. Therefore, by applying Theorem \ref{thm:MAIN'} and a compactness theorem for Ricci flows without injectivity radius estimates (see Theorem \ref{thm:Lottcompactness}), we obtain that the Ricci flow solutions $(M,g^{(n)}(t))$ converge locally in the $\mathcal{C}^{\infty}$-topology to the product $\mathbb{R}^s \times (M/\mathsf{T},h(t))$. Here, $\mathsf{T}^s \subset \mathsf{N}_{\mathsf{G}}(\mathsf{H})/\mathsf{H}$ is a compact torus acting freely on $M$ and $h(t)$ is a non-collapsed ancient homogeneous Ricci flow on the quotient $M/\mathsf{T}$. By \cite[Theorem 4.2]{BLS19}, $h(t)$ emanates from an Einstein metric $\check{g}_o$ on $M/\mathsf{T}$ with positive scalar curvature. This allows us to conclude that $\mathbb{R}^s \times (B,\check{g}(t))$ is a blow-down limit of the original Ricci flow $(M,g(t))$, where $\check{g}(t)$ is the ancient solution on $B$ satisfying $\check{g}(-1) = \check{g}_o$. We refer to Theorem \ref{thm:last} for more details. \smallskip

We now make some remarks about our main results. An important feature of Theorem \ref{thm:MAIN'} is that the collapse occurs in the direction of a compact torus $\mathsf{T} \subset N_{\mathsf{G}}(\mathsf{H})/\mathsf{H}$. As a consequence, the blow-down limit in Theorem \ref{thm:MAIN} is in fact complete. Another consequence is that $M$ admits no collapsed ancient $\mathsf{G}$-homogeneous Ricci flows if the group of $\mathsf{G}$-equivariant diffeomorphisms $\mathsf{N}_{\mathsf{G}}(\mathsf{H}) / \mathsf{H}$ is finite. This was already known by \cite{BWZ04,Ped19} (see also \cite[Remark 5.3]{BLS19} and \cite[Theorem A]{KPS25}). This implies, for example, that collapsed ancient homogeneous K{\"ahler}-Ricci flows cannot exist on compact manifolds with finite fundamental group (see, e.g., \cite[Sect.\ 8.88]{Bes08}). \smallskip

Let us also recall some related results on ancient homogeneous Ricci flows which have already appeared in the literature. In \cite[Theorem A]{PedSb22}, a general existence theorem was established, which generalizes all previously known homogeneous examples. This construction produces ancient solutions on a compact manifold $M$ emanating from a (not necessarily unstable) Einstein metric on the base $M/\mathsf{T}$ of a principal torus bundle. Moreover, in \cite[Theorem D]{KPS25}, the Gromov--Hausdorff convergence to an Einstein metric on the base of a torus bundle was proved under additional assumptions. \smallskip 

Finally, we emphasize that the compactness assumption in Theorem \ref{thm:MAIN} plays a fundamental role. However, to the best of our knowledge, no examples are known of ancient homogeneous Ricci flows on non-compact manifolds other than the trivial products of a flat factor with a compact ancient solution. The only general fact available so far is that, by \cite[Th{\'e}or{\`e}me 2]{BeBe78}, any non-compact ancient homogeneous Ricci flow on an aspherical manifold must be flat, and therefore trivial. \medskip

The paper is organized as follows. Section \ref{sect:prel} collects background material on ancient homogeneous Ricci flows and on convergence in the absence of injectivity radius estimates. In Section \ref{sect:GHconv}, we establish a technical result on the Gromov--Hausdorff convergence of compact homogeneous spaces (see Theorem \ref{thm:equivGH1}) and we state a symmetrization result for collapsed homogeneous metrics (see Theorem \ref{thm:equivGH2}). In Section \ref{sect:smoothconv}, we prove a local $\mathcal{C}^{\infty}$-convergence result in the presence of collapse (see Theorem \ref{thm:equivGH3}), from which Theorem \ref{thm:MAIN'} follows. In Section \ref{sect:main}, we apply Theorem \ref{thm:MAIN'} to prove Theorem \ref{thm:last}, which in turn implies Theorem \ref{thm:MAIN}. Finally, Appendix \ref{sect:appendix} contains a proof of the rigidity of locally homogeneous gradient Ricci solitons, and Appendix \ref{sect:appendix2} provides a proof of Theorem \ref{thm:equivGH2}.
\medskip

\noindent {\itshape Acknowledgements.\ } The authors would like to thank Christoph B{\"o}hm, Ramiro Lafuente and John Lott for helpful discussions.

\medskip
\section{Preliminaries}
\label{sect:prel} \setcounter{equation} 0

Throughout this paper, we adopt the following notation for the geometric data associated with a Riemannian manifold $(M,g)$:
\begin{itemize}
\item[$\bcdot$] $D^g$ denotes the Levi-Civita connection;
\item[$\bcdot$] $|\cdot|_g$ denotes the induced norm on tensors;
\item[$\bcdot$] $\mathrm{inj}_p(M,g)$ denotes the injectivity radius at $p$, $\mathrm{inj}(M,g)$ denotes the injectivity radius;
\item[$\bcdot$] $\mathtt{d}_g$ denotes the Riemannian distance function;
\item[$\bcdot$] $\mathscr{B}^M_g(p,r) = \{x \in M : \mathtt{d}_g(x,p) < r\}$ denotes the geodesic ball centered at $p \in M$ of radius $r>0$.
\end{itemize}
In the special case where $(M,g) = (\mathbb{R}^m, \langle \cdot, \cdot \rangle)$ is the flat Euclidean space, we denote by $|\cdot|$ the induced norm and by $\mathscr{B}^{\mathbb{R}^m}(p,r)$ the corresponding geodesic balls. \smallskip

\subsection{Ancient homogeneous Ricci flows} \hfill \par

Let $M^m = \mathsf{G}/\mathsf{H}$ be the quotient of a compact, connected, non-abelian Lie group $\mathsf{G}$ by a closed subgroup $\mathsf{H} \subset \mathsf{G}$. Up to passing to a finite cover, we can assume that $\mathsf{H}$ is also connected. We further assume that the $\mathsf{G}$-action on $M$ is {\it almost effective}, i.e., the set of elements of $\mathsf{G}$ acting trivially on $M$ is discrete. We denote by $\mathscr{M}^{\mathsf{G}}$ the space of $\mathsf{G}$-invariant metrics on $M$. \smallskip

The Ricci flow  
$$
\partial_t g(t) = -2\,\mathrm{Ric}(g(t))
$$
preserves $\mathsf{G}$-invariance and therefore reduces to a finite-dimensional dynamical system on $\mathscr{M}^{\mathsf{G}}$. By \cite[Theorem 2]{Boe15}, any homogeneous Ricci flow on the compact homogeneous space $M$ becomes extinct in finite forward time, so the long-time behavior analysis of the flow reduces to the study of ancient solutions as $t \to -\infty$. Up to a time shift, we assume that any ancient solution has maximal interval of existence $(-\infty,0)$.

By \cite[Corollary 2]{BLS19}, any ancient homogeneous Ricci flow $g(t)$ is of Type-I. More precisely, there exists a constant $C>1$, depending only on the dimension $m$ and on $\mathrm{scal}(g(-1))$, such that
\begin{equation} \label{eq:typeI}
\big|\mathrm{Rm}(g(t))\big|_{g(t)} \cdot |t| \in [C^{-1},C] \quad \text{for all $t \leq -1$} \,\, .
\end{equation}
Moreover, the behavior of its scalar curvature $\mathrm{scal}(g(t))$ determines the geometry of $g(t)$ as $t \to -\infty$. Indeed, the quantity
$$
F(g) \coloneqq \mathrm{vol}(M,g)^{\frac2m}\mathrm{scal}(g)
$$
is scale-invariant and monotonically non-decreasing along the homogeneous Ricci flow:
$$
\frac{\mathrm{d}}{\mathrm{d} t}F(g(t)) = 2\mathrm{vol}(M,g(t))^{\frac2m} \Big( \big|\mathrm{Ric}(g(t))\big|_{g(t)}^2 -\tfrac1m\mathrm{scal}(g(t))^2\Big) \geq 0 \,\, ,
$$
with equality if and only if $g(t)$ evolves through Einstein metrics. Moreover, if $g(t)$ is ancient, then $\mathrm{scal}(g(t))>0$ (see, e.g., \cite[p.\ 102]{CLN06}). Thus, two distinct scenarios arise (see, e.g., \cite[Remark 5.3]{BLS19}):
\begin{itemize}
\item[$a)$] if $F(g(t))$ is bounded away from zero, then $g(t)$ is non-collapsed;
\item[$b)$] if $F(g(t)) \to 0$ as $t \to -\infty$, then the solution is collapsed.
\end{itemize}
\smallskip

Here, we recall that an ancient Ricci flow $g(t)$ is said to be {\it non-collapsed} if the injectivity radius of the curvature normalized solution $\tfrac1{|t|}g(t)$ is bounded away from zero as $t \to -\infty$, and it is said to be {\it collapsed} otherwise (see, e.g., \cite[Definition 6.44, Lemma 6.54]{CCGGIIKLLN07}). \smallskip

In case $a)$, the solution emanates from an unstable Einstein metric on the same manifold (see \cite[Theorem 4.2]{BLS19}). Case $b)$ is however not fully understood, except in special instances (see \cite[Theorem D]{KPS25}). By \cite{BWZ04} and \cite[Remark 5.13]{BLS19}, $M$ admits a collapsed ancient solution only if it is the total space of a (possibly locally defined) homogeneous torus bundle
$$
\mathsf{T} \to \mathsf{G}/\mathsf{H} \to \mathsf{G}/\mathsf{H}\mathsf{T} \,\, .
$$
Here, $\mathsf{T}$ is a connected, abelian subgroup of a compact complement of $\mathsf{H}$ inside its normalizer $N_{\mathsf{G}}(\mathsf{H})$. Recall that the quotient $N_{\mathsf{G}}(\mathsf{H})/\mathsf{H}$ acts on $M = \mathsf{G}/\mathsf{H}$ by right multiplication and is naturally isomorphic to the group of $\mathsf{G}$-equivariant diffeomorphisms of $M$ (see \cite[Corollary I.4.3]{Br72}). Moreover, the identity component $(N_{\mathsf{G}}(\mathsf{H})/\mathsf{H})_0$ can be identified with a compact subgroup of $\mathsf{G}$ (see \cite[Lemma 4.7]{Boe04}). For later use, we make the following observation.

\begin{proposition} \label{prop:flatorb}
Let $g$ be a $\mathsf{G}$-invariant metric on $M = \mathsf{G}/\mathsf{H}$ and let $\mathsf{T} \subset N_{\mathsf{G}}(\mathsf{H})/\mathsf{H}$ be a torus. Then, the right $\mathsf{T}$-orbits on $M$ are flat with respect to the metric induced by $g$.
\end{proposition}

\begin{proof}
We notice that the right $\mathsf{T}$-orbit through the origin $e\mathsf{H} \cdot \mathsf{T}$ in $M$ coincides with the corresponding left orbit $\mathsf{T} \cdot e\mathsf{H}$. The left orbit $\mathsf{T} \cdot e\mathsf{H}$ with the metric induced by $g$ is a Riemannian homogeneous space $(\mathsf{H}\mathsf{T}/\mathsf{H} = \mathsf{T},g|_{\mathsf{T} \cdot e\mathsf{H}})$, and so it is flat because $\mathsf{T}$ is abelian. Moreover, since the left action of $\mathsf{G}$ is transitive and commutes with the right action of $\mathsf{T}$, it follows that the right $\mathsf{T}$-orbits are pairwise isometric via the ambient action of $\mathsf{G}$ on $(M,g)$. Therefore, every right $\mathsf{T}$-orbit in $(M,g)$ is flat.
\end{proof}

\subsection{Convergence in the absence of injectivity radius estimates} \label{subsect:conv} \hfill \par

Following \cite{BLS19, Ped22}, we recall the following definition.

\begin{definition} \label{def:geom-mod}
A {\it geometric model} consists of a smooth locally homogeneous Riemannian manifold $(\mathbb{E}, \tilde{g})$ and a distinguished point $o \in \mathbb{E}$, called the {\it origin}, such that
$$
\mathscr{B}^{\mathbb{E}}_{\tilde{g}}(o,\pi) = \mathbb{E} \,\, , \quad |{\rm sec}(\tilde{g})| \leq 1 \,\, , \quad {\rm inj}_o(\mathbb{E}, \tilde{g})=\pi \,\, .
$$
\end{definition}

By \cite[Theorem A]{Ped22}, any locally homogeneous Riemannian manifold $(M,g)$ with $|{\rm sec}(g)| \leq 1$ is locally isometric to a geometric model $(\mathbb{E}, \tilde{g})$, which is unique up to isometry. In this case, we refer to $(\mathbb{E}, \tilde{g})$ as the {\it geometric model of $(M,g)$.} If $(M,g)$ is globally homogeneous, its geometric model can be constructed explicitly. Indeed, fix a point $p \in M$, choose a $g_p$-orthonormal frame $u: \mathbb{R}^m \to T_pM$, and define
\begin{equation} \label{eq:geom-mod}
\mathbb{E} \coloneqq \mathscr{B}^{\mathbb{R}^m}(0,\pi) = \{x \in \mathbb{R}^m : |x| < \pi \} \,\, , \quad o \coloneqq 0 \in \mathbb{R}^m \,\, , \quad f \coloneqq \mathrm{Exp}^g_p \circ u|_{\mathbb{E}} \,\, , \quad \tilde{g} \coloneqq \varphi^*g \,\, ,
\end{equation}
where $\mathrm{Exp}^g_p$ denotes the exponential map at $p$ of $(M,g)$. Then, $(\mathbb{E}, \tilde{g})$ is a geometric model with origin $o$ and $\varphi: (\mathbb{E}, \tilde{g}) \to (M, g)$ is a local isometry. We remark that, by Definition \ref{def:geom-mod}, we may always assume that the underlying manifold $\mathbb{E}$ of a geometric model is the standard Euclidean ball $\mathbb{E} = \mathscr{B}^{\mathbb{R}^m}(0,\pi)$ and that $g$ is radially isometric to $\langle \cdot, \cdot \rangle$.
\smallskip

Geometric models provide a natural framework for defining the convergence of homogeneous spaces when no injectivity radius estimates are available (see \cite[Theorem B, Corollary 3.8]{Ped22}). This approach was initiated in the works of Glickenstein and Lott for the study of limits of collapsed Ricci flows. More precisely, in \cite[Theorem 3]{Gli03}, a limit flow was constructed on a ball in a single tangent space. Later, in \cite[Theorem 1.4]{Lott07}, Riemannian groupoids were introduced to assemble such local limits across different tangent spaces. 

The following definition makes precise the notion of convergence of geometric models (c.f. \cite[Definition 2.1]{Ped22}).

\begin{definition}
Let $(M^{(n)},g^{(n)})$ be a sequence of Riemannian homogeneous spaces and denote by $(\mathbb{E},\tilde{g}^{(n)})$ the corresponding sequence of geometric models. The manifolds $(M^{(n)},g^{(n)})$ are said to {\it converge locally in the $\mathcal{C}^{\infty}$-topology} to a Riemannian locally homogeneous space $(M^{(\infty)},g^{(\infty)})$ if the manifolds $(\mathbb{E},\tilde{g}^{(n)})$ converge to the geometric model $(\mathbb{E},\tilde{g}^{(\infty)})$ of $(M^{(\infty)},g^{(\infty)})$ in the following sense: for every $0 < \delta < \pi$, there exist an integer $\bar{n} = \bar{n}(\delta) \in \mathbb{N}$ and smooth embeddings $\phi^{(n)}: \overline{\mathscr{B}^{\mathbb{E}}_{\tilde{g}^{(\infty)}}(o,\pi-\delta)} \to \mathbb{E}$, for all $n \geq \bar{n}$, such that $\phi^{(n)}(o) = o$, and $(\phi^{(n)})^*\tilde{g}^{(n)} \to \tilde{g}^{(\infty)}$ smoothly on $\overline{\mathscr{B}^{\mathbb{E}}_{\tilde{g}^{(\infty)}}(o,\pi-\delta)}$ as $n \to +\infty$.
\end{definition}

\begin{remark}
The above notion of convergence is closely related to the convergence in the sense of Riemannian groupoids introduced in \cite{Lott07}. Indeed, the {\`e}tale groupoid $\mathcal{G} = (\mathcal{G}_0, \mathcal{G}_1)$ of a Riemannian manifold $(M,g)$, with $|{\rm sec}(g)| \leq 1$, consists of the disjoint union $\mathcal{G}_0$ of a set of tangent balls of radius $\pi$, with the corresponding pullback metrics via the exponential map, and a set $\mathcal{G}_1$ of morphisms containing gluing information (see \cite[Example 5.7]{Lott07}). Convergence of a sequence $\mathcal{G}^{(n)}=(\mathcal{G}^{(n)}_0,\mathcal{G}^{(n)}_1)$ requires both smooth convergence of the Riemannian manifolds $\mathcal{G}^{(n)}_0$, and Hausdorff convergence of the sets $\mathcal{G}^{(n)}_1$ (see \cite[Definition 5.8]{Lott07}). If $(M,g)$ is homogeneous, then $\mathcal{G}_0$ is the disjoint union of a set of locally homogenous Riemannian manifolds that are isometric to the geometric model of $(M,g)$. Moreover, the local convergence in the $\mathcal{C}^{\infty}$-topology of a sequence $(M^{(n)},g^{(n)})$ is equivalent to the smooth convergence of the corresponding $\mathcal{G}^{(n)}_0$. 
\end{remark}

The corresponding notion of convergence for ancient Ricci flows reads as follows.

\begin{definition} \label{def:convloc}
Let $(M^{(n)},g^{(n)}(t))$ be a sequence of ancient homogeneous Ricci flows and denote by $(\mathbb{E},\tilde{g}^{(n)}_*)$ the geometric models of $(M^{(n)},g^{(n)}(-1))$, with local isometry $f^{(n)}: (\mathbb{E},\tilde{g}^{(n)}_*\big) \to \big(M,g^{(n)}(-1))$ as in \eqref{eq:geom-mod}. The flows $(M^{(n)},g^{(n)}(t))$ are said to {\it converge locally in the $\mathcal{C}^{\infty}$-topology} to an ancient locally homogeneous Ricci flow $(M^{(\infty)},g^{(\infty)}(t))$ if the following condition is satisfied. For every $0 < \delta < \pi$, there exist an integer $\bar{n} = \bar{n}(\delta) \in \mathbb{N}$ and smooth embeddings $\phi^{(n)}: \overline{\mathscr{B}^{\mathbb{E}}_{g^{(\infty)}(-1)}(o,\pi-\delta)} \to \mathbb{E}$, for all $n \geq \bar{n}$, such that: \begin{itemize}
\item[$i)$] $\phi^{(n)}(o) = o$ and, as $n \to +\infty$, the metrics
$$
\mathrm{d}t^2 + (\phi^{(n)})^*\big((f^{(n)})^*(g^{(n)}(t))\big)
$$
converge smoothly to a limit metric
$$
\mathrm{d}t^2 + \tilde{g}^{(\infty)}(t)
$$
on $I \times \overline{\mathscr{B}^{\mathbb{E}}_{g^{(\infty)}(-1)}(o,\pi-\delta)}$, for every compact interval $I \subset (-\infty, 0)$;
\item[$ii)$] $(\mathbb{E}, \tilde{g}^{(\infty)}(t_o))$ is locally isometric to $(M^{(\infty)},g^{(\infty)}(t_o))$ for all $t_o < 0$.
\end{itemize}
\end{definition}

For blow-downs along an ancient homogeneous Ricci flow, which is necessarily Type-I by \cite[Corollary 2]{BLS19}, the compactness results established by Glickenstein and Lott imply the following.

\begin{theorem}[c.f. \cite{Gli03}, Theorem 3; \cite{Lott07}, Theorem 1.4; \cite{BLS19}, Corollary 2] \label{thm:Lottcompactness}
Let $(M,g(t))$ be an ancient homogeneous Ricci flow, let $\tau^{(n)} \to +\infty$ be a sequence of positive numbers and set $g^{(n)}(t) \coloneqq \frac1{\tau^{(n)}}g(\tau^{(n)}t)$. Then the sequence of blow-downs $(M,g^{(n)}(t))$ subconverges locally in the $\mathcal{C}^{\infty}$-topology to a limit locally homogeneous Ricci flow as $n \to +\infty$.
\end{theorem}

\subsection{Preparatory definitions for convergence} \label{subsect:notation} \hfill \par

Let $(M^m,g)$ be a compact Riemannian manifold. If $T$ is a tensor field on $M$, we set
$$
\|T\|_{\mathcal{C}^0(M,g)} \coloneqq \max_{M} |T|_g \,\, , \quad \|T\|_{\mathcal{C}^k(M,g)} \coloneqq \sum_{i=0}^k  \Big\| \big|(D^g)^iT\big|_g \Big\|_{\mathcal{C}^0(M,g)} \,\, .
$$
Notice that, although the $\mathcal{C}^k$-norm $\|\cdot\|_{\mathcal{C}^k(M,g)}$ depends on the Riemannian metric $g$, the topology it induces on the space of tensor fields does not, since $M$ is compact.

\begin{definition} \label{def:boundedgeom}
Let $\{\Gamma_k\}$ be a sequence of positive numbers. If $(M,g)$ is a compact Riemannian manifold, we say that {\it $(M,g)$ has $\{\Gamma_k\}$-bounded geometry} if $\|\mathrm{Rm}(g)\|_{\mathcal{C}^k(M,g)} \leq \Gamma_k$ for every $k \geq 0$, where $\mathrm{Rm}(g)$ denotes the Riemann curvature tensor of $(M,g)$.
\end{definition}

We recall that, given a smooth map $f: (M, g) \to (N, h)$ between compact Riemannian manifolds, the {\it $k$-th covariant derivative of $f$} is the $f^*TN$-valued $(0,k)$-tensor field on $M$ defined recursively by
$$\begin{gathered}
\nabla^0f \coloneqq f \,\, , \quad \nabla^1f \coloneqq \mathrm{d}f \,\, , \\
\nabla^{k+1}f(X_0,X_1,{\dots},X_k) \coloneqq D^{h,f}_{f_*X_0}\big(\nabla^kf(X_1,{\dots},X_k)\big) -\sum_{i=1}^k\nabla^kf(X_1,{\dots},D^g_{X_0}X_i,{\dots},X_k) \,\, ,
\end{gathered}$$
where $D^{h,f}$ denotes the connection induced on $f^*TN$ by the Levi-Civita connection $D^h$ of $(N,h)$. Notice that the metrics $g$ and $h$ allow to measure the norm $|\nabla^kf|_{g,h}$ of $\nabla^kf$, and so we set
$$
\|f\|_{\mathcal{C}^k(M,g,h)} \coloneqq \sum_{i=0}^k  \Big\| \big|\nabla^kf\big|_{g,h} \Big\|_{\mathcal{C}^0(M,g)} \,\, .
$$
Again, although the $\mathcal{C}^k$-norm $\|f\|_{\mathcal{C}^k(M,g,h)}$ depends on both the Riemannian metrics $g$ and $h$, the topology it induces on the space of smooth maps from $M$ to $N$ does not, since both $M$ and $N$ are compact.

\begin{definition}
Let $\{\Gamma_k\}$ be a sequence of positive numbers. If  $f: (M, g) \to (N, h)$ is a smooth map between compact Riemannian manifolds, we say that {\it $f$ is $\{\Gamma_k\}$-bounded} if $\|f\|_{\mathcal{C}^k(M,g,h)} \leq \Gamma_k$ for every $k \geq 0$.
\end{definition}

For later use, we remark the following fact about convergence of maps.

\begin{remark} \label{rem:convf}
Let $f^{(n)}: (M^{(n)}, g^{(n)}) \to (N^{(n)}, h^{(n)})$ be a sequence of smooth maps between compact Riemannian manifolds. Assume that both $(M^{(n)}, g^{(n)})$ and $(N^{(n)}, h^{(n)})$ converge in the $\mathcal{C}^{\infty}$-topology to some compact Riemannian manifolds $(M^{(\infty)}, g^{(\infty)})$ and $(N^{(\infty)}, h^{(\infty)})$, respectively. Namely, there exist sequences of diffeomorphisms $\phi_1^{(n)} : M^{(\infty)} \to M^{(n)}$ and $\phi_2^{(n)} : N^{(\infty)} \to N^{(n)}$ such that $(\phi_1^{(n)})^*g^{(n)} \to g^{(\infty)}$ and $(\phi_2^{(n)})^*h^{(n)} \to h^{(\infty)}$ smoothly. Assume also that $f^{(n)}$ is $\{\Gamma_k\}$-bounded for all $n \in \mathbb{N}$, for some sequence of positive numbers $\{\Gamma_k\}$. Then, up to passing to a subsequence, by the Arzel\`a--Ascoli Theorem there exists a smooth map $f^{(\infty)}: (M^{(\infty)}, g^{(\infty)}) \to (N^{(\infty)}, h^{(\infty)})$ such that $(\phi_2^{(n)})^{-1} \circ f^{(n)} \circ \phi_1^{(n)} \to f^{(\infty)}$ smoothly as $n \to +\infty$. In this situation, we say that the maps $f^{(n)}$ {\it converge in the $\mathcal{C}^{\infty}$-topology} to $f^{(\infty)}$.
\end{remark}

We also recall the following two definitions. The first one is the notion of {\it Gromov--Hausdorff approximation}, that quantifies the distance between compact Riemannian manifolds in the Gromov--Hausdorff topology.

\begin{definition}
Let $\epsilon >0$ and let $f: (M^m,g) \to (B^{m-s},\check{g})$ be a (not necessarily continuous) map between compact Riemannian manifolds. The map $f$ is called {\it $\epsilon$-Gromov--Hausdorff approximation} if: \begin{itemize}
\item[$\bcdot$] $\big|\mathtt{d}_{g}(p, p') -\mathtt{d}_{\check{g}}(f(p), f(p'))\big| <  \epsilon$ for all $p, p' \in M$;
\item[$\bcdot$] for every $b \in B$ there exists $p \in M$ such that $\mathtt{d}_{\check{g}}(f(p), b) < \epsilon$.
\end{itemize}
\end{definition}

The second definition quantifies the failure of a given map to be a Riemannian submersion.

\begin{definition}
Let $\epsilon >0$ and let $f: (M^m,g) \to (B^{m-s},\check{g})$ be a smooth map between compact Riemannian manifolds. The map $f$ is called: \begin{itemize}
\item[$i)$] {\it $\epsilon$-almost Riemannian submersion} if
$$
(1- \epsilon) |v|_{g} \leq \big|\mathrm{d}f|_p(v)\big|_{\check{g}} \leq (1+ \epsilon) |v|_{g}
$$
for all $p \in M$ and $v \in {\rm ker}({\rm d}f|_p)^{\perp_{g}} \subset T_pM$;
\item[$ii)$] {\it $\epsilon$-almost Riemannian submersion in the $\mathcal{C}^2$-sense} \cite{NaTian18} if it is an $\epsilon$-almost Riemannian submersion and
$$
\big|\nabla^2f(v,w)\big|_{g,\check{g}} < \epsilon |v|_{g} |w|_{g}
$$
for all $p \in M$ and $v, w \in {\rm ker}({\rm d}f|_p)^{\perp_{g}} \subset T_pM$.
\end{itemize}
\end{definition}

\medskip
\section{Equivariant Gromov--Hausdorff convergence of compact homogeneous spaces}
\label{sect:GHconv} \setcounter{equation} 0

In this section, we characterize the Riemannian collapse of compact homogeneous spaces having uniformly bounded geometry and uniformly bounded diameter. For the sake of notation, given $D>0$ and a sequence $\{\Gamma_k\}$ of positive numbers, we set
\begin{equation} \label{eq:MGbound}
\mathscr{M}^{\mathsf{G}}\big(D,\{\Gamma_k\}\big) \coloneqq \big\{ g \in \mathscr{M}^{\mathsf{G}} : \text{$|\mathrm{sec}(g)| \leq 1$, $\mathrm{diam}(M,g) \leq D$, $(M,g)$ has $\{\Gamma_k\}$-bounded geometry} \big\} \,\, .
\end{equation}
Here the notion of $\{\Gamma_k\}$-bounded geometry has been introduced in Definition \ref{def:boundedgeom}. Moreover, we observe that the condition $|\mathrm{sec}(g)| \leq 1$ is just a convenient normalization, not an extra condition. \smallskip

We now establish the following precompactness result for $\mathscr{M}^{\mathsf{G}}\big(D,\{\Gamma_k\}\big)$.

\begin{theorem} \label{thm:equivGH1}
Let $M^m = \mathsf{G}/\mathsf{H}$ be a homogeneous space, with $\mathsf{H}$ and $\mathsf{G}$ compact, connected Lie groups. Let also $D >0$ be a constant and $\{\Gamma_k\}$ a sequence of positive numbers. Then, for every sequence $\{g^{(n)}\} \subset \mathscr{M}^{\mathsf{G}}\big(D,\{\Gamma_k\}\big)$, there exist, up to passing to a subsequence, a torus $\mathsf{T}^s \subset N_{\mathsf{G}}(\mathsf{H})/\mathsf{H}$, with $0 \leq s \leq m$, a $\mathsf{G}$-invariant Riemannian metric $\check{g}^{(\infty)}$ on $M/\mathsf{T}$, a sequence of positive numbers $\{\widetilde{\Gamma}_k\}$, depending only on the data $\big(m,s, D, \{\Gamma_k\},\mathrm{inj}(M/\mathsf{T},\check{g}^{(\infty)}),\mathsf{G}\big)$ and a sequence $\{\epsilon^{(n)}\}$ of positive numbers, with $\epsilon^{(n)} \to 0$ as $n \to +\infty$, such that, for all $n \in \mathbb{N}$, the canonical projection
$$
\pi: (M,g^{(n)}) \to (M/\mathsf{T},\check{g}^{(\infty)})
$$
is an $\epsilon^{(n)}$-Gromov--Hausdorff approximation, an $\epsilon^{(n)}$-almost Riemannian submersion in the $\mathcal{C}^2$-sense, and is $\{\widetilde{\Gamma}_k\}$-bounded.
\end{theorem}

\begin{proof}
We divide the proof into three steps.
\smallskip

\noindent {\it Step 1: The manifolds $(M,g^{(n)})$ subconverge in the Gromov--Hausdorff topology to a smooth $\mathsf{G}$-homogenous space $(B,\check{g}^{(\infty)})$.}

Since $(M,g^{(n)})$ is a sequence of compact Riemannian $m$-manifolds with uniformly bounded curvature and diameter, up to passing to a subsequence, it converges in the Gromov--Hausdorff topology to a compact Alexandrov space $(B,\mathtt{d}^{(\infty)})$ of dimension $m-s$, for some $0 \leq s \leq m$, with curvature $\geq -1$ and diameter ${\rm diam}(B,\mathtt{d}^{(\infty)}) \leq D$ (see, e.g., \cite[Theorem 10.7.2]{BBI01} and the discussion following it). By \cite[Corollary, Sect.\ 6]{Gr81}, $(B,\mathtt{d}^{(\infty)})$ is homogeneous, i.e., for any pair of points $b_1, b_2 \in B$ there exists an isometry of $(B,\mathtt{d}^{(\infty)})$ mapping $b_1$ to $b_2$. Moreover, by \cite[Theorem 7]{Ber89}, $(B,\mathtt{d}^{(\infty)})$ is smooth, namely, $B$ admits a smooth manifold structure and a Riemannian metric $\check{g}^{(\infty)}$ such that $\mathtt{d}_{\check{g}^{(\infty)}} = \mathtt{d}^{(\infty)}$. If $s = 0$, this concludes the proof of the theorem. Therefore, from now on we assume $s >0$.

The compact Lie group $\mathsf{G}$ acts by isometries on each $(M,g^{(n)})$ and the diameter ${\rm diam}(M,g^{(n)})$ is uniformly bounded. Therefore, by \cite[Theorem 6.9]{Fu90}, up to passing to a subsequence, the convergence of $(M,g^{(n)})$ to $(B,\check{g}^{(\infty)})$ is {\it $\mathsf{G}$-equivariant} in the sense of \cite[Definition 6.8]{Fu90}. In particular, there exists an isometric action $\Theta : \mathsf{G} \times B \to B$ of the same Lie group $\mathsf{G}$ on the limit Riemannian manifold $(B,\check{g}^{(\infty)})$. We remark that $\Theta(a,\_)$ is smooth for all $a \in \mathsf{G}$ by the Myers-Steenrod Theorem \cite[Theorem 8]{MS39}, and so it follows that $\Theta$ is smooth by \cite[Theorem 4]{BM45}. For the sake of notation, we just write $\Theta(a,b) = a \cdot b$ for every $a \in \mathsf{G}$ and $b \in B$.
\smallskip

\noindent {\it Step 2: There exists a torus $\mathsf{T} \subset N_{\mathsf{G}}(\mathsf{H})/\mathsf{H}$ such that $B = M/\mathsf{T}$.}

By the equivariant version of Fukaya's Fiber Bundle Theorem \cite[Theorem 12.1, Theorem 12.1']{Fu90}, there exist a sequence $\{\epsilon^{(n)}\}$ of positive numbers, with $\epsilon^{(n)} \to 0$, and a sequence of smooth maps $f^{(n)}: (M,g^{(n)}) \to (B,g^{(\infty)})$ satisfying, for every $n \in \mathbb{N}$:
\begin{itemize}
\item[$\bcdot$] $f^{(n)}$ is a $\mathsf{G}$-equivariant bundle projection with infranil fibers;
\item[$\bcdot$] $f^{(n)}$ is an $\epsilon^{(n)}$-Gromov--Hausdorff approximation;
\item[$\bcdot$] $f^{(n)}$ is an $\epsilon^{(n)}$-almost Riemannian submersion.
\end{itemize}
Moreover, by the enhanced version \cite[Theorem 3.1, Theorem 3.2]{NaTian18}, up to changing $\epsilon^{(n)}$, there exists a sequence of positive numbers $\{\widetilde{\Gamma}_k\}$, depending only on the data $\big(m,s, D, \{\Gamma_k\},\mathrm{inj}(B,\check{g}^{(\infty)}),\mathsf{G}\big)$, such that the maps $f^{(n)}$ can be assumed to satisfy the following further properties for every $n \in \mathbb{N}$:
\begin{itemize}
\item[$\bcdot$] $f^{(n)}$ is $\{\widetilde{\Gamma}_k\}$-bounded;
\item[$\bcdot$] $f^{(n)}$ is an $\epsilon^{(n)}$-almost Riemannian submersion in the $\mathcal{C}^2$-sense.
\end{itemize}
Since $\mathsf{G}$ acts transitively on $M$ and the maps $f^{(n)}$ are $\mathsf{G}$-equivariant and surjective, it follows that the action $\Theta$ is transitive. Consider the sequence of points $\check{o}^{(n)} \coloneqq f^{(n)}(e\mathsf{H}) \in B$ and the corresponding isotropy groups $\mathsf{K}^{(n)} \coloneqq \{a \in \mathsf{G} : a \cdot \check{o}^{(n)} = \check{o}^{(n)}\}$. Notice that $\mathsf{H} \subset \mathsf{K}^{(n)}$ for all $n \in \mathbb{N}$, since
$$
\check{o}^{(n)} = f^{(n)}(e\mathsf{H}) = f^{(n)}(h \cdot e\mathsf{H}) = h \cdot f^{(n)}(e\mathsf{H}) = h \cdot \check{o}^{(n)}
$$
for every $h \in \mathsf{H}$. Moreover, the fibers of the $\mathsf{G}$-equivariant projection $f^{(n)}: M \to B$ are diffeomorphic to the quotient $\mathsf{K}^{(n)}/\mathsf{H}$. Since infranil manifolds are aspherical and $\mathsf{K}^{(n)}$ is compact, it follows that $\mathsf{K}^{(n)}/\mathsf{H}$ is diffeomorphic to a torus for every $n \in \mathbb{N}$.

Up to passing to a subsequence, we can assume that there exists $\check{o} \in B$ such that $\check{o}^{(n)} \to \check{o}$ as $n \to +\infty$. Define the isotropy group
$$
\mathsf{K} \coloneqq \{a \in \mathsf{G} : a \cdot \check{o} = \check{o}\}
$$
and observe that $B$ is $\mathsf{G}$-equivariantly diffeomorphic to the quotient $\mathsf{G}/\mathsf{K}$ via the map $a\mathsf{K} \mapsto a \cdot \check{o}$.

Fix a sequence $\{\alpha^{(n)}\} \subset \mathsf{G}$ such that $\alpha^{(n)} \cdot \check{o} = \check{o}^{(n)}$ for every $n \in \mathbb{N}$. Then, it follows that $\mathsf{K}^{(n)} = \alpha^{(n)}.\mathsf{K}.(\alpha^{(n)})^{-1}$ for every $n \in \mathbb{N}$. Moreover, we can choose $\alpha^{(n)}$ so that $\alpha^{(n)} \to e$ as $n \to +\infty$. Therefore, $(\alpha^{(n)})^{-1}.\mathsf{H}.\alpha^{(n)} \subset \mathsf{K}$ for all $n \in \mathbb{N}$, and the map $f^{(n)} : \mathsf{G}/\mathsf{H} \to \mathsf{G}/\mathsf{K}$ takes the form $f^{(n)}(a\mathsf{H}) = a\alpha^{(n)}\mathsf{K}$ for every $a \in \mathsf{G}$. Moreover, it follows that the quotient $\mathsf{K}/\mathsf{H}$ is diffeomorphic to a torus, i.e., $\mathsf{K} = \mathsf{H}\mathsf{T}$ for some torus $\mathsf{T}$ inside a compact complement of $\mathsf{H}$ within its normalizer $N_{\mathsf{G}}(\mathsf{H})$.
\smallskip

\noindent {\it Step 3: The canonical projection $\pi: M \to M/\mathsf{T}$ has the required properties.}

It remains to show that
$$
\pi: (M,g^{(n)}) \to (M/\mathsf{T},\check{g}^{(\infty)})
$$
is a $2\epsilon^{(n)}$-Gromov--Hausdorff approximation, a $2\epsilon^{(n)}$-almost Riemannian submersion in the $\mathcal{C}^2$-sense and $\{2\widetilde{\Gamma}_k\}$-bounded for every $n \in \mathbb{N}$.

Fix a biinvariant metric $Q$ on $\mathsf{G}$. Notice that, by compactness of $\mathsf{G}$ and by continuity of $\Theta$, we can pass to a subsequence so that
$$
\mathtt{d}_{\check{g}^{(\infty)}}\big(e\mathsf{K}, (a^{-1}.(\alpha^{(n)})^{-1}.a.\alpha^{(n)}) \cdot e\mathsf{K}\big) \leq \epsilon^{(n)} \quad \text{for all $a \in \mathsf{G}$ and $n \in \mathbb{N}$} \,\, .
$$
Therefore, for every $a_1, a_2 \in \mathsf{G}$, we get
\begin{align*}
\big|\mathtt{d}_{g^{(n)}}(a_1\mathsf{H},a_2\mathsf{H}) -&\mathtt{d}_{\check{g}^{(\infty)}}(\pi(a_1\mathsf{H}),\pi(a_2\mathsf{H}))\big| \\
&\leq \big|\mathtt{d}_{g^{(n)}}(a_1\mathsf{H},a_2\mathsf{H}) -\mathtt{d}_{\check{g}^{(\infty)}}(f^{(n)}(a_1\mathsf{H}),f^{(n)}(a_2\mathsf{H}))\big| \\
&\qquad \qquad +\big|\mathtt{d}_{\check{g}^{(\infty)}}(f^{(n)}(a_1\mathsf{H}),f^{(n)}(a_2\mathsf{H})) -\mathtt{d}_{\check{g}^{(\infty)}}(\pi(a_1\mathsf{H}),\pi(a_2\mathsf{H}))\big| \\
&\leq \epsilon^{(n)} + \big|\mathtt{d}_{\check{g}^{(\infty)}}(a_1.\alpha^{(n)}.e\mathsf{K},a_2.\alpha^{(n)}.e\mathsf{K}) -\mathtt{d}_{\check{g}^{(\infty)}}(a_1.e\mathsf{K},a_2.e\mathsf{K})\big| \\
&\leq \epsilon^{(n)} + \big|\mathtt{d}_{\check{g}^{(\infty)}}(e\mathsf{K},(\alpha^{(n)})^{-1}.a_1^{-1}.a_2.\alpha^{(n)}.e\mathsf{K}) -\mathtt{d}_{\check{g}^{(\infty)}}(e\mathsf{K},a_1^{-1}.a_2.e\mathsf{K})\big| \\
&\leq \epsilon^{(n)} + \mathtt{d}_{\check{g}^{(\infty)}}(e\mathsf{K},(a_1^{-1}.a_2)^{-1}.(\alpha^{(n)})^{-1}.a_1^{-1}.a_2.\alpha^{(n)}.e\mathsf{K}) \\
&\leq 2\epsilon^{(n)} \,\, ,
\end{align*}
which shows that $\pi$ is a $2\epsilon^{(n)}$-Gromov--Hausdorff approximation. \smallskip

By compactness of $M$ and by the properties of $f^{(n)}$ (see Step 2 above), in order to prove that $\pi$ is a $2\epsilon^{(n)}$-almost Riemannian submersion in the $\mathcal{C}^2$-sense and $\{2\widetilde{\Gamma}_k\}$-bounded for every $n \in \mathbb{N}$, it is sufficient to show that every point $p^* \in M$ admits a neighborhood $\mathscr{U} \subset M$ with the following property: there exists a sequence of smooth functions $\rho^{(n)} : \mathscr{U} \to \mathsf{G}$ such that $f^{(n)}(p) = \rho^{(n)}(p).\pi(p)$ for every $p \in \mathscr{U}$ that converges smoothly to the constant map $\imath : \mathscr{U} \to \mathsf{G}$, $\imath(p) \coloneqq e$, as $n \to +\infty$.

Fix then a point $p^* \in M$ and $a^* \in \mathsf{G}$ such that $p^* = a^*\mathsf{H}$. Take a slice $\mathscr{S} \subset \mathsf{G}$ through $a^*$ for the right action of $\mathsf{H}$ by right translations and set $\mathscr{U} \coloneqq \{a\mathsf{H} : a \in \mathscr{S}\} \subset M$. This defines a smooth local section
$$
\sigma: \mathscr{U} \to \mathsf{G} \,\, , \quad \sigma(a\mathsf{H}) \coloneqq a
$$
for the principal bundle $\mathsf{G} \to \mathsf{G}/\mathsf{H}$. Define now the maps
\begin{align}
\gamma^{(n)} : \mathsf{G} \to \mathsf{G} \,\, &, \quad \gamma^{(n)}(a) \coloneqq a . \alpha^{(n)} . a^{-1} \,\, , \label{gamma_n} \\ 
\rho^{(n)} : \mathscr{U} \to \mathsf{G} \,\, &, \quad \rho^{(n)} \coloneqq \gamma^{(n)} \circ \sigma \,\, . \label{eq:rho_n}
\end{align}
Then, by definition,
\begin{equation} \label{eq:phi-pi}
f^{(n)}(p) = \rho^{(n)}(p).\pi(p) \quad \text{for every $p \in \mathscr{U}$} \,\, .
\end{equation}
Since the metric $Q$ is biinvariant, by \eqref{gamma_n} we get
$$
\mathtt{d}_{Q}(\gamma^{(n)}(a), e) = \mathtt{d}_{Q}(\alpha^{(n)}, e) \quad \text{for all $a \in \mathsf{G}$ and $n \in \mathbb{N}$} \,\, .
$$
Therefore, since $\alpha^{(n)} \to e$ as $n \to +\infty$, it follows from that $\gamma^{(n)}$ converges in the $\mathcal{C}^{0}$-topology to the constant map $\mathsf{G} \to \mathsf{G}$, $a \mapsto e$, as $n \to +\infty$. In fact, it is routine to check that the convergence takes place in the $\mathcal{C}^{\infty}$-topology. Therefore, it follows by \eqref{eq:rho_n} that $\rho^{(n)}$ converges to the constant map
$$
\imath : \mathscr{U} \to \mathsf{G} \,\, , \quad \imath(p) \coloneqq e
$$
in the $\mathcal{C}^{\infty}$-topology as $n \to +\infty$.
\end{proof}

As a byproduct of the proof of Theorem \ref{thm:equivGH1}, we obtain the following convergence result which does not assume higher order derivative bounds.

\begin{corollary}
Let $M^m=\mathsf{G}/\mathsf{H}$ be a compact, connected homogeneous space and let $D>0$ be a constant. Then every sequence $\{g^{(n)}\} \subset \mathscr{M}^{\mathsf{G}}$ satisfying $|{\rm sec}(g^{(n)})|\leq 1$ and ${\rm diam}(M,g^{(n)}) \leq D$ for all $n\in\mathbb{N}$ subconverges in the Gromov--Hausdorff topology to an invariant metric on $\mathsf{G}/\mathsf{K}$, for some closed intermediate subgroup $\mathsf{H}\subset \mathsf{K} \subset \mathsf{G}$ such that $\mathsf{K}/\mathsf{H}$ is a torus.
\end{corollary}

\begin{proof}
Consider a sequence of metrics $\{g^{(n)}\} \subset \mathscr{M}^{\mathsf{G}}$, with $|{\rm sec}(g^{(n)})| \leq 1$ and ${\rm diam}(M, g^{(n)}) \leq D$. Even though $\{g^{(n)}\}$ has no higher order uniform bounds, the claim follows by suitably modifying the proof of Theorem \ref{thm:equivGH1}. Indeed, Step 1 applies without any changes. Moreover, in Step 2, we can still apply results from \cite{Fu90}, without the enhanced version given in \cite{NaTian18}, to show that $B = \mathsf{G}/\mathsf{K}$ for some closed intermediate subgroup $\mathsf{H} \subset \mathsf{K} \subset \mathsf{G}$ such that $\mathsf{K}/\mathsf{H}$ is a torus.
\end{proof}

We now recall the following technical result from \cite{ChFuGro92}, which has found numerous applications in the literature on collapsed Riemannian manifolds (see, e.g., \cite[Theorem 2.1]{Rong96}, \cite[Theorem 2]{Lott02}, \cite[Theorem 2]{Lott'02}, \cite[Section 2]{Sin16}, \cite[Section 4]{NaTian18}). It establishes estimates for the averaging process of collapsing metrics under the right action of $\mathsf{T}$.

\begin{theorem}[c.f.\ \cite{ChFuGro92}, Proposition 4.9] \label{thm:equivGH2}
Let $M^m = \mathsf{G}/\mathsf{H}$ be a homogeneous space, with $\mathsf{H}$ and $\mathsf{G}$ compact, connected Lie groups. Let also $D >0$ be a constant and $\{\Gamma_k\}$ a sequence of positive numbers. Fix a sequence $\{g^{(n)}\} \subset \mathscr{M}^{\mathsf{G}}\big(D,\{\Gamma_k\}\big)$ and let the data $\big(\mathsf{T}, \check{g}^{(\infty)},\{\epsilon^{(n)}\}\big)$ be as in the conclusion of Theorem \ref{thm:equivGH1}. Consider the averaged metrics
\begin{equation} \label{eq:averaging}
g^{(n)}_{\mathsf{T}} \coloneqq \int_{\mathsf{T}} \big(f^*g^{(n)}\big)\, {\rm d} \lambda(f) \,\, ,
\end{equation}
where $\lambda$ denotes the Haar measure on $\mathsf{T}$. Then, there exists a sequence $\{C_k\}$ of positive numbers, which depends only on the data $\big(m,s, D, \{\Gamma_k\},\mathrm{inj}(M/\mathsf{T},\check{g}^{(\infty)}),\mathsf{G}\big)$ and not on $\{\epsilon^{(n)}\}$, such that
\begin{equation} \label{eq:estaveraging}
\big\|g^{(n)} -g^{(n)}_{\mathsf{T}}\big\|_{\mathcal{C}^k} \leq C_k\, \epsilon^{(n)} \quad \text{ for every $k, n \in \mathbb{N}$} \,\, .
\end{equation}
\end{theorem}

By means of Theorem \ref{thm:equivGH2}, whenever we are in the setting of Theorem \ref{thm:equivGH1}, we can assume without loss of generality that the metrics $g^{(n)}$ are also right $\mathsf{T}$-invariant. Since this plays an important role in the proof of our main results, we outline the proof of Theorem \ref{thm:equivGH2} in Appendix \ref{sect:appendix2}. 

\medskip
\section{Geometry of collapsed compact homogeneous spaces}
\label{sect:smoothconv} \setcounter{equation} 0

In this section we use the results of Section \ref{sect:GHconv} to characterize the local $\mathcal{C}^{\infty}$-limit of sequences in $\mathscr{M}^{\mathsf{G}}\big(D,\{\Gamma_k\}\big)$. The result below, together with Theorem \ref{thm:equivGH1}, completes the proof of Theorem \ref{thm:MAIN'}.

\begin{theorem} \label{thm:equivGH3}
Let $M^m = \mathsf{G}/\mathsf{H}$ be a homogeneous space, with $\mathsf{H}$ and $\mathsf{G}$ compact, connected Lie groups. Let also $D >0$ be a constant and $\{\Gamma_k\}$ a sequence of positive numbers. Fix a sequence $\{g^{(n)}\} \subset \mathscr{M}^{\mathsf{G}}\big(D,\{\Gamma_k\}\big)$ and let the data $\big(\mathsf{T}, \check{g}^{(\infty)}\big)$ be as in the conclusion of Theorem \ref{thm:equivGH1}. Then, up to passing to a subsequence, $(M,g^{(n)})$ converges locally in the $\mathcal{C}^{\infty}$-topology to $\mathbb{R}^s \times (M/\mathsf{T}, \check{g}^{(\infty)})$ as $n \to +\infty$.
\end{theorem}

We notice that, if ${\rm vol}(M,g^{(n)})$ is uniformly bounded away from zero, then $\mathsf{T} = \{e\}$, $s=0$ and Theorem \ref{thm:equivGH3} is a straightforward consequence of the Cheeger--Gromov Compactness Theorem. Hence, we restrict our attention to the case ${\rm vol}(M,g^{(n)}) \to 0$ as $n \to +\infty$. \smallskip

Before proceeding with the proof, we introduce some notation. Let $\mathfrak{g}$ denote the Lie algebra of $\mathsf{G}$ and $\mathfrak{h}$ the Lie algebra of $\mathsf{H}$. Fix an ${\rm Ad}(\mathsf{G})$-invariant Euclidean inner product $Q$ on $\mathfrak{g}$, and let $\mathfrak{m}$ be the $Q$-orthogonal complement of $\mathfrak{h}$ in $\mathfrak{g}$. Via the infinitesimal $\mathsf{G}$-action, we identify $\mathfrak{m} \simeq T_{e\mathsf{H}}M$, and hence we identify any $\mathsf{G}$-invariant tensor field on $M$ with the corresponding ${\rm Ad}(\mathsf{H})$-invariant tensor on $\mathfrak{m}$. Moreover, the identity component $(N_{\mathsf{G}}(\mathsf{H})/\mathsf{H})_0$ is identified with a compact subgroup of $\mathsf{G}$ whose Lie algebra is the ${\rm Ad}(\mathsf{H})$-trivial submodule
$$
\mathfrak{m}_0 \coloneqq \big\{X \in \mathfrak{m} : [\mathfrak{h},X] = \{0\}\big\}
$$
(see, e.g., \cite[Lemma 4.27]{Boe04}). In order to prove Theorem \ref{thm:equivGH3}, we will need a number of intermediate results. To avoid repetition, we collect some background notations, assumptions and facts in the following Remark.

\begin{remark} \label{rem:thm41}
Let $\{g^{(n)}\} \subset \mathscr{M}^{\mathsf{G}}\big(D,\{\Gamma_k\}\big)$ be a sequence such that ${\rm vol}(M,g^{(n)}) \to 0$ as $n \to +\infty$ and let the data $\big(\mathsf{T},\check{g}^{(\infty)},\{\epsilon^{(n)}\}\big)$ be as in Theorem \ref{thm:equivGH1}. Let us also assume that each metric $g^{(n)}$ is ${\rm Ad}(\mathsf{T})$-invariant. We observe the following facts.
\begin{itemize}
\item[$i)$] The fact that, without loss of generality, we can prove Theorem \ref{thm:equivGH3} by choosing a sequence of metrics $g^{(n)}$ that are also ${\rm Ad}(\mathsf{T})$-invariant follows by Theorem \ref{thm:equivGH2}.
\item[$ii)$] By ${\rm Ad}(\mathsf{T})$-invariance, each $g^{(n)}$ induces a $\mathsf{G}$-invariant Riemannian metric $\check{g}^{(n)}$ on $B = \mathsf{G}/\mathsf{H}\mathsf{T}$ such that the $\mathsf{T}$-principal bundle
\begin{equation} \label{eq:sub(n)}
\pi: (\mathsf{G}/\mathsf{H},g^{(n)}) \to (\mathsf{G}/\mathsf{H}\mathsf{T},\check{g}^{(n)})
\end{equation}
is a Riemannian submersion for all $n \in \mathbb{N}$.
\item[$iii)$] The Lie algebra $\mathfrak{t}$ of $\mathsf{T}$ is an abelian subalgebra of $\mathfrak{m}_0$, with $\ker(\mathrm{d}\pi|_{e\mathsf{H}}) = \mathfrak{t}$. We denote by $\hat{g}^{(n)}$ the inner product induced by $g^{(n)}$ on $\mathfrak{t}$, and by $\mathfrak{b}^{(n)}$ the $g^{(n)}$-orthogonal complement of $\mathfrak{t}$ in $\mathfrak{m}$.
\end{itemize}
\end{remark}

We begin by proving some properties of the sequence of Riemannian submersions defined in \eqref{eq:sub(n)}.

\begin{lemma} \label{lem:bn}
Consider the framework of Remark \ref{rem:thm41}. Then, the following claims hold. \begin{itemize}
\item[(i)] The subspace $\mathfrak{b}^{(n)}$ is a reductive complement for $\mathfrak{h}+\mathfrak{t}$ in $\mathfrak{g}$, i.e., $[\mathfrak{h} + \mathfrak{t}, \mathfrak{b}^{(n)}] \subset \mathfrak{b}^{(n)}$.
\item[(ii)] The fibers of $\pi$ are totally geodesic for all $n \in \mathbb{N}$.
\item[(iii)] $\lim_{n \to +\infty} |V|_{g^{(n)}} = 0$ for every $V \in \mathfrak{t}$.
\end{itemize}
\end{lemma}

\begin{proof}
Since $\mathsf{T}$ is abelian and $g^{(n)}$ is $\mathrm{Ad}(\mathsf{T})$-invariant, it follows that for all $V_1, V_2 \in \mathfrak{t}$ and $X^{(n)} \in \mathfrak{b}^{(n)}$
$$
g^{(n)}([V_1,X^{(n)}],V_2) = -g^{(n)}([V_1,V_2],X^{(n)}) = 0 \,\, ,
$$
and so $[\mathfrak{t},\mathfrak{b}^{(n)}] \subset \mathfrak{b}^{(n)}$. This proves claim $(i)$. Moreover, by \cite[Lemma 7.27]{Bes08}, we have
$$
2g^{(n)}\big(D^{g^{(n)}}_{V_1}V_2,X^{(n)}\big) = g^{(n)}([V_1,V_2],X^{(n)}) +g^{(n)}([V_1,X^{(n)}],V_2) +g^{(n)}([V_2,X^{(n)}],V_1) = 0 \,\, ,
$$
which proves claim $(ii)$. In particular, the intrinsic distance of each fiber coincides with the distance in the ambient space $(M,g^{(n)})$ for all $n \in \mathbb{N}$. Moreover, since $\pi: (M,g^{(n)}) \to (B,\check{g}^{(\infty)})$ is $\epsilon^{(n)}$-Gromov--Hausdorff approximation, it follows that $\mathrm{diam}(\pi^{-1}(b),g^{(n)}) \to 0$ as $n \to +\infty$ for all $b \in B$, which concludes the proof of claim $(iii)$.
\end{proof}

\begin{remark} \label{rem:Tinvmetrics}
The space of left $\mathsf{G}$-invariant and right $\mathsf{T}$-invariant metrics on $M=\mathsf{G}/\mathsf{H}$ is identified with the space $S^2_+(\mathfrak{m})^{\mathrm{Ad}(\mathsf{H}\mathsf{T})}$ of positive-definite, $\mathrm{Ad}(\mathsf{H}\mathsf{T})$-invariant symmetric bilinear forms on $\mathfrak{m}$. Each element $g \in S^2_+(\mathfrak{m})^{\mathrm{Ad}(\mathsf{H}\mathsf{T})}$ induces a metric $\check{g}$ on $B =\mathsf{G}/\mathsf{H}\mathsf{T}$ so that the homogeneous principal torus bundle $\pi: (M,g) \to (B, \check{g})$ is a Riemannian submersion. Consequently, there is a one-to-one correspondence between elements $g \in S^2_+(\mathfrak{m})^{\mathrm{Ad}(\mathsf{H}\mathsf{T})}$ and triples $(\hat{g},\mathfrak{b},\check{g})$ given by a flat metric $\hat{g} \in S^2_+(\mathfrak{t})$, an ${\rm Ad}(\mathsf{T})$-invariant subspace $\mathfrak{b} \subset \mathfrak{m}$ and a horizontal metric $\check{g} \in S^2_+(\mathfrak{b})^{\mathrm{Ad}(\mathsf{H}\mathsf{T})}$. By the proof of Lemma \ref{lem:bn}, the toral fibers of $\pi$ are totally geodesic with respect to any metric $g \in S^2_+(\mathfrak{m})^{\mathrm{Ad}(\mathsf{H}\mathsf{T})}$.
\end{remark}

We now turn to the limiting behavior of the Riemannian submersions \eqref{eq:sub(n)} and establish the following two intermediate results.

\begin{proposition} \label{prop:GHbases}
Consider the framework of Remark \ref{rem:thm41}. Then, the manifolds $(B,\check{g}^{(n)})$ converge in the Gromov--Hausdorff topology to $(B,\check{g}^{(\infty)})$ as $n \to +\infty$.
\end{proposition}

\begin{proof}
Notice that, since \eqref{eq:sub(n)} is a Riemannian submersion, for any pair of points $b_0, b_1 \in B$, there exist $p_0 \in \pi^{-1}(b_0)$ and $p_1 \in \pi^{-1}(b_1)$ such that $\mathtt{d}_{g^{(n)}}(p_0, p_1) =\mathtt{d}_{\check{g}^{(n)}}(b_0, b_1)$. Since $\pi: (M,g^{(n)}) \to (B,\check{g}^{(\infty)})$ is an $\epsilon^{(n)}$-Gromov--Hausdorff approximation, this implies that
$$\begin{aligned}
\big|\mathtt{d}_{\check{g}^{(n)}}(b_0, b_1) -\mathtt{d}_{\check{g}^{(\infty)}}(b_0, b_1)\big| = \big|\mathtt{d}_{g^{(n)}}(p_0, p_1) -\mathtt{d}_{\check{g}^{(\infty)}}(\pi(p_0), \pi(p_1))\big|\leq \epsilon^{(n)}
\end{aligned}$$
and so this concludes the proof.
\end{proof}

\begin{lemma} \label{lem:binfty}
Consider the framework of Remark \ref{rem:thm41}. Then, the sequence $\{\mathfrak{b}^{(n)}\}$ subconverges to a limit reductive complement $\mathfrak{b}^{(\infty)}$ for $\mathfrak{h}+\mathfrak{t}$ in $\mathfrak{g}$.
\end{lemma}

\begin{proof}
Since the Grassmannian of $(m{-}s)$-dimensional subspaces in $\mathfrak{m}$ is compact, we can pass to a subsequence such that $\mathfrak{b}^{(n)}$ converges to a limit $(m{-}s)$-dimensional subspace $\mathfrak{b}^{(\infty)}$ as $n \to +\infty$. Moreover, since $\mathfrak{b}^{(n)}$ is $\mathrm{Ad}(\mathsf{H}\mathsf{T})$-invariant for all $n \in \mathbb{N}$, it follows that $\mathfrak{b}^{(\infty)}$ is $\mathrm{Ad}(\mathsf{H}\mathsf{T})$-invariant as well. It remains to show that $\mathfrak{t} \cap \mathfrak{b}^{(\infty)} = \{0\}$. Recall that, by Theorem \ref{thm:equivGH1},
\begin{equation} \label{eq:almsub}
(1- \epsilon^{(n)}) \big|v^{(n)}\big|_{g^{(n)}} \leq \big|\mathrm{d}\pi|_{e\mathsf{H}}(v^{(n)})\big|_{\check{g}^{(\infty)}} \leq (1+ \epsilon^{(n)}) \big|v^{(n)}\big|_{g^{(n)}} \quad \text{ for every $v^{(n)} \in \mathfrak{b}^{(n)}$ .}
\end{equation}
We now prove the following two claims. \smallskip

\noindent{\it Claim 1: For any $v^{(\infty)} \in \mathfrak{b}^{(\infty)}$, we have $v^{(\infty)} \in \mathfrak{t} \cap \mathfrak{b}^{(\infty)}$ if and only if, for every sequence $v^{(n)} \in \mathfrak{b}^{(n)}$ with $v^{(n)} \to v^{(\infty)}$, one has $|v^{(n)}|_{g^{(n)}} \to 0$ as $n \to +\infty$.}

By construction, for any $v^{(\infty)} \in \mathfrak{b}^{(\infty)}$ there exists a sequence $v^{(n)} \in \mathfrak{b}^{(n)}$ such that $v^{(n)} \to v^{(\infty)}$. If $|v^{(n)}|_{g^{(n)}} \to 0$ for every such sequence, then, by \eqref{eq:almsub}, we have
$$
\big|\mathrm{d}\pi|_{e\mathsf{H}}(v^{(\infty)})\big|_{\check{g}^{(\infty)}} = 0
$$
and so $v^{(\infty)} \in \mathfrak{t}$. Conversely, if there exists a sequence $v^{(n)} \in \mathfrak{b}^{(n)}$ with $v^{(n)} \to v^{(\infty)}$ such that for each $n$, we have $|v^{(n)}|_{g^{(n)}} \geq \delta > 0$, then, by \eqref{eq:almsub}, we obtain
$$
\big|\mathrm{d}\pi|_{e\mathsf{H}}(v^{(\infty)})\big|_{\check{g}^{(\infty)}} \geq \delta
$$
and so $v^{(\infty)} \not\in \mathfrak{t}$. This concludes the proof of Claim 1. \smallskip

\noindent{\it Claim 2: If there exists a sequence $v^{(n)} \in \mathfrak{b}^{(n)}$ with $|v^{(n)}|_{Q}=1$ such that $|v^{(n)}|_{g^{(n)}}\to 0$ as $n \to +\infty$, then ${\rm vol}(B,\check{g}^{(n)}) \to 0$ as $n \to +\infty$.}

Let $\{v_i^{(n)}\}$ be a $Q$-orthonormal basis for $\mathfrak{b}^{(n)}$ and assume that $|v_1^{(n)}|_{g^{(n)}} \to 0$ as $n \to +\infty$. By compactness, we can assume that $\{v_i^{(n)}\}$ converges to a $Q$-orthonormal basis $\{v_i^{(\infty)}\}$ for $\mathfrak{b}^{(\infty)}$ as $n \to +\infty$. Moreover, by \eqref{eq:almsub}, we have
$$
|v_i^{(n)}|_{g^{(n)}} \leq \frac1{1-\epsilon^{(n)}} \big|{\rm d}\pi|_{e\mathsf{H}}(v^{(n)}_i)\big|_{\check{g}^{(\infty)}} \to \big|{\rm d}\pi|_{e\mathsf{H}}(v^{(\infty)}_i)\big|_{\check{g}^{(\infty)}}
$$
and so there exists $C >0$ such that $|v_i^{(n)}|_{g^{(n)}} \leq C$ for all $1 \leq i \leq m-s$ and $n \in \mathbb{N}$. In particular, since $|\cdot |_{g^{(n)}} = |\cdot |_{\check{g}^{(n)}}$ on $\mathfrak{b}^{(n)}$, we obtain
$$
\big|\check{g}^{(n)}(v_i^{(n)},v_j^{(n)})\big| \leq |v_i^{(n)}|_{g^{(n)}} |v_j^{(n)}|_{g^{(n)}} \leq C^2
$$
and, for $i = 1$,
$$
\big|\check{g}^{(n)}(v_1^{(n)},v_j^{(n)})\big| \leq C\, |v_1^{(n)}|_{g^{(n)}} \to 0 \,\, .
$$
Therefore $\det\big(\check{g}^{(n)}(v_i^{(n)},v_j^{(n)})\big) \to 0$ as $n \to +\infty$ and this concludes the proof of Claim 2. \smallskip

Finally, assume by contradiction that there exists $v^{(\infty)} \in \mathfrak{t} \cap \mathfrak{b}^{(\infty)}$ with $|v^{(\infty)}|_Q=1$. By Claim 1, there exists a sequence $v^{(n)} \in \mathfrak{b}^{(n)}$ such that $|v^{(n)}|_Q=1$, $v^{(n)} \to v^{(\infty)}$ and $|v^{(n)}|_{g^{(n)}} \to 0$ as $n \to +\infty$. By Claim 2, this implies $\mathrm{vol}(B,\check{g}^{(n)}) \to 0$. Recall that, by \cite[Theorem 0.1]{Col97}, the volume function is continuous under the Gromov--Hausdorff convergence in the presence of a uniform Ricci curvature lower bound. Moreover, by claim (ii) in Lemma \ref{lem:bn} and \cite[(9.36c)]{Bes08}, the manifolds $(B,\check{g}^{(n)})$ have Ricci curvature uniformly bounded from below. Therefore, by Proposition \ref{prop:GHbases}, we deduce that $\mathrm{vol}(B,\check{g}^{(\infty)}) = 0$, yielding a contradiction.
\end{proof}

As a consequence of Lemma \ref{lem:binfty}, we can identify $T_{e\mathsf{H}\mathsf{T}}B$ with $\mathfrak{b}^{(\infty)}$. Hence the limit metric $\check{g}^{(\infty)}$ on $B$ induces an $\mathrm{Ad}(\mathsf{H}\mathsf{T})$-invariant inner product on $\mathfrak{b}^{(\infty)}$, which we denote again by the same symbol. Under the identification in Remark \ref{rem:Tinvmetrics}, the following corollary holds.

\begin{corollary} \label{cor:conv}
Consider the framework of Remark \ref{rem:thm41}, Remark \ref{rem:Tinvmetrics} and Lemma \ref{lem:binfty}. Then, the sequence $\{g^{(n)}\}$ of metrics $g^{(n)} \simeq (\hat{g}^{(n)},\mathfrak{b}^{(n)},\check{g}^{(n)})$ converges in the standard Euclidean topology of $S^2(\mathfrak{m})^{\mathrm{Ad}(\mathsf{H}\mathsf{T})}$ to the degenerate metric $(0,\mathfrak{b}^{(\infty)},\check{g}^{(\infty)})$ as $n \to +\infty$.
\end{corollary}

\begin{proof}
We first notice that the differential $\mathrm{d}\pi|_{e\mathsf{H}}$ of the natural projection $\pi: M \to B$ corresponds to the canonical projection $\mathrm{pr}_{\mathfrak{b}^{(\infty)}}: \mathfrak{m} \to \mathfrak{b}^{(\infty)}$ induced by the direct sum decomposition $\mathfrak{m} = \mathfrak{t} \oplus \mathfrak{b}^{(\infty)}$. Moreover, $\hat{g}^{(n)} \to 0$ in $S^2(\mathfrak{t})$ by claim (iii) in Lemma \ref{lem:bn} and $\mathfrak{b}^{(n)} \to \mathfrak{b}^{(\infty)}$ by Lemma \ref{lem:binfty}. Finally, paraphrasing \eqref{eq:almsub}, we obtain
$$
(1-\epsilon^{(n)}) |v^{(n)}|_{\check{g}^{(n)}} \leq \big|\mathrm{pr}_{\mathfrak{b}^{(\infty)}}(v^{(n)})\big|_{\check{g}^{(\infty)}} \leq (1+\epsilon^{(n)}) |v^{(n)}|_{\check{g}^{(n)}} \quad \text{ for every $v^{(n)} \in \mathfrak{b}^{(n)}$} \,\, ,
$$
from which the claim follows.
\end{proof}

Let $A^{(n)}$ denote the O'Neill tensor \cite{ONeill66} of the Riemannian submersion \eqref{eq:sub(n)}.

\begin{lemma} \label{lem:Akth}
Consider the framework of Remark \ref{rem:thm41}. Then, the $k$-covariant derivative $(D^{g^{(n)}})^k\!A^{(n)}$ of the O'Neill tensor $A^{(n)}$ vanishes asymptotically as $n \to +\infty$ for all $k \geq 0$.
\end{lemma}

\begin{proof}
Let $\{e^{(n)}_{\alpha}\}$ be a $g^{(n)}$-orthonormal basis for $\mathfrak{b}^{(n)}$. By Corollary \ref{cor:conv}, up to passing to a subsequence, we may assume that $\{e^{(n)}_{\alpha}\}$ converges to a $\check{g}^{(\infty)}$-orthonormal basis $\{e^{(\infty)}_{\alpha}\}$ for $\mathfrak{b}^{(\infty)}$. For the sake of notation, for every $n \in \mathbb{N} \cup \{\infty\}$, we denote by $\mathrm{pr}^{(n)}_{\mathfrak{t}} : \mathfrak{m} \to \mathfrak{t}$ the projection induced by the direct sum decomposition $\mathfrak{m} = \mathfrak{t} \oplus \mathfrak{b}^{(n)}$. By \cite[Proposition 9.24]{Bes08}, we then have

\begin{equation} \label{eq:An}
\big|A^{(n)}\big|_{g^{(n)}}^2 = \frac14 \sum_{\alpha, \beta} \big|\mathrm{pr}^{(n)}_{\mathfrak{t}}\big([e^{(n)}_{\alpha},e^{(n)}_{\beta}]\big)\big|_{\hat{g}^{(n)}}^2 \,\, .
\end{equation}

By Corollary \ref{cor:conv}, it follows that $\mathrm{pr}^{(n)}_{\mathfrak{t}} \to \mathrm{pr}^{(\infty)}_{\mathfrak{t}}$ in the space of linear endomorphisms of $\mathfrak{m}$ with the Euclidean topology. Since $[e^{(n)}_{\alpha},e^{(n)}_{\beta}] \to [e^{(\infty)}_{\alpha},e^{(\infty)}_{\beta}]$ and by Lemma \ref{lem:bn}, $\hat{g}^{(n)} \to 0$ as $n \to +\infty$, we obtain
\begin{equation} \label{eq:pr_t}
\big|\mathrm{pr}^{(n)}_{\mathfrak{t}}\big([e^{(n)}_{\alpha},e^{(n)}_{\beta}]\big)\big|_{\hat{g}^{(n)}} \to 0 \quad \text{as $n \to +\infty$}\,\, .
\end{equation}

By \cite[Section 9.32]{Bes08}, the only non-vanishing components of the first covariant derivative $D^{g^{(n)}}\!A^{(n)}$ are given by
\begin{equation} \label{eq:A'}
\big(D^{g^{(n)}}_{X^{(n)}}A^{(n)}\big)_U = -A^{(n)}_{A^{(n)}_{X^{(n)}}U} \,\, , \quad \big(D^{g^{(n)}}_{X^{(n)}}A^{(n)}\big)_{Y^{(n)}} = A^{(n)}_{X^{(n)}}A^{(n)}_{Y^{(n)}} -A^{(n)}_{Y^{(n)}}A^{(n)}_{X^{(n)}}
\end{equation}
for $X^{(n)} \in \mathfrak{b}^{(n)}$, $U \in \mathfrak{t}$. Hence, from \eqref{eq:A'}, it follows that for every $k \in \mathbb{N}$ there exists $C(m,k) >0$ such that
$$
\big|(D^{g^{(n)}})^k\!A^{(n)}\big|_{g^{(n)}} \leq C(m,k) \big|A^{(n)}\big|_{g^{(n)}}^{k+1} \,\,.
$$
By \eqref{eq:An} and \eqref{eq:pr_t}, this completes the proof.
\end{proof}

By means of Lemma \ref{lem:Akth}, we can strengthen the Gromov--Hausdorff convergence of Proposition \ref{prop:GHbases}  to $\mathcal{C}^{\infty}$ convergence as follows.

\begin{corollary} \label{cor:CGbases}
Consider the framework of Remark \ref{rem:thm41}. Then, the manifolds $(B,\check{g}^{(n)})$ subconverge in the $\mathcal{C}^{\infty}$-topology to $(B,\check{g}^{(\infty)})$ as $n \to +\infty$.
\end{corollary}

\begin{proof}
By Proposition \ref{prop:GHbases}, the manifolds $(B,\check{g}^{(n)})$ converge in the Gromov--Hausdorff topology to $(B,\check{g}^{(\infty)})$, which is a Riemannian manifold of the same dimension. Moreover, since $(M,g^{(n)})$ has $\{\Gamma_k\}$-bounded geometry, by Lemma \ref{lem:Akth} and \cite[Theorem 9.28]{Bes08}, it follows that there exists a sequence of positive numbers $\{\check{\Gamma}_k\}$, depending only on $\big(m,s,\{\Gamma_k\}\big)$, such that $(B,\check{g}^{(n)})$ has $\{\check{\Gamma}_k\}$-bounded geometry for all $n \in \mathbb{N}$. Thus, the claim follows from the Cheeger--Gromov Compactness Theorem.
\end{proof}

We are ready to prove Theorem \ref{thm:equivGH3}.

\begin{proof}[Proof of Theorem \ref{thm:equivGH3}]
Let us denote by $\big(\mathbb{E},\tilde{g}^{(n)}\big)$ the geometric model of $(M,g^{(n)})$, with origin $o^{(n)}$ and local isometry $\varphi^{(n)}: (\mathbb{E},\tilde{g}^{(n)}) \to (M,g^{(n)})$ as in \eqref{eq:geom-mod}. By \cite[Corollary 3.8]{Ped22}, up to passing to a subsequence, the manifolds $(\mathbb{E}, \tilde{g}^{(n)})$ converge in the pointed $\mathcal{C}^{\infty}$-topology to a limit geometric model $(\mathbb{E}, \tilde{g}^{(\infty)})$, with origin $o^{(\infty)}$, as $n \to +\infty$. We need to show that $(\mathbb{E}, \tilde{g}^{(\infty)})$ is locally isometric to $\mathbb{R}^s \times (M/\mathsf{T}, \check{g}^{(\infty)})$. \smallskip

To begin with, we notice that the local submersions
$$
\widetilde{\pi}^{(n)} : (\mathbb{E},\tilde{g}^{(n)}) \to (B,\check{g}^{(n)}) \,\, , \quad \widetilde{\pi}^{(n)} \coloneqq \pi \circ \varphi^{(n)}
$$
are $\{\Gamma_k\}$-bounded, that $(\mathbb{E},\tilde{g}^{(n)})$ converges in the pointed $\mathcal{C}^{\infty}$-topology to $(\mathbb{E}, \tilde{g}^{(\infty)})$ and that $(B,\check{g}^{(n)})$ converges in the pointed $\mathcal{C}^{\infty}$-topology to $(B,\check{g}^{(\infty)})$ (see Corollary \ref{cor:CGbases}). Therefore, up to passing to a subsequence, it follows by the Arzel\`a--Ascoli Theorem that $\widetilde{\pi}^{(n)}$ converges in the $\mathcal{C}^{\infty}$-topology to a smooth map
\begin{equation} \label{eq:limitpi}
\widetilde{\pi}^{(\infty)} : (\mathbb{E},\tilde{g}^{(\infty)}) \to (B,\check{g}^{(\infty)})
\end{equation}
(see Remark \ref{rem:convf}). Next, notice that $\widetilde{\pi}^{(n)}$ is a Riemannian submersion, and so, for every $n \in \mathbb{N}$, the linear map
$$
{\rm d}\widetilde{\pi}^{(n)}|_{o^{(n)}} : (T_{o^{(n)}}\mathbb{E},\tilde{g}^{(n)}_{o^{(n)}}) \to (T_{\check{o}}B,\check{g}^{(n)}_{\check{o}})
$$
admits a right-inverse $S^{(n)}$ with the following properties: \begin{itemize}
\item[$\bcdot$] the image of $S^{(n)}$ is ${\rm Im}(S^{(n)}) = {\rm ker}({\rm d}\widetilde{\pi}^{(n)}|_{o^{(n)}})^{\perp_{\tilde{g}^{(n)}}}$;
\item[$\bcdot$] $S^{(n)}$ preserves the inner products, i.e., $(S^{(n)})^*(\tilde{g}^{(n)}_{o^{(n)}}) = \check{g}^{(n)}_{\check{o}}$.
\end{itemize}
Therefore, since $(\mathbb{E},\tilde{g}^{(n)})$ converges in the pointed $\mathcal{C}^{\infty}$-topology to $(\mathbb{E}, \tilde{g}^{(\infty)})$ and $(B,\check{g}^{(n)})$ converges in the pointed $\mathcal{C}^{\infty}$-topology to $(B,\check{g}^{(\infty)})$, up to passing to a subsequence, it follows that $S^{(n)}$ converges in the standard Euclidean topology to a linear map
$$
S^{(\infty)} : (T_{\check{o}}B,\check{g}^{(\infty)}_{\check{o}}) \to (T_{o^{(\infty)}}\mathbb{E},\tilde{g}^{(\infty)}_{o^{(\infty)}})
$$
 with the following properties: \begin{itemize}
\item[$\bcdot$] $S^{(\infty)}$ is the right inverse of ${\rm d}\widetilde{\pi}^{(\infty)}|_{o^{(\infty)}}$;
\item[$\bcdot$] $S^{(\infty)}$ preserves the inner products, i.e., $(S^{(\infty)})^*(\tilde{g}^{(\infty)}_{o^{(\infty)}}) = \check{g}^{(\infty)}_{\check{o}}$.
\end{itemize}
This proves that $\widetilde{\pi}^{(\infty)}$ is a Riemannian submersion. \smallskip

By claim ii) in Lemma \ref{lem:bn}, it follows that the fibers of $\widetilde{\pi}^{(\infty)}$ are totally geodesic. Moreover, Lemma \ref{lem:Akth} implies that the O'Neill tensor of \eqref{eq:limitpi} vanishes, hence the horizontal distribution of $\widetilde{\pi}^{(\infty)}$ is integrable. This implies that $(\mathbb{E},\tilde{g}^{(\infty)})$ is locally isometric to the Riemannian product of the fiber $F^{(\infty)} \coloneqq (\widetilde{\pi}^{(\infty)})^{-1}(\check{o})$ and the base $(B,\check{g}^{(\infty)})$. It remains to show that the fiber $F^{(\infty)}$ is flat. \smallskip

Since $\widetilde{\pi}^{(n)}$ converges to $\widetilde{\pi}^{(\infty)}$ in the $\mathcal{C}^{\infty}$-topology, it follows that the fibers $F^{(n)} \coloneqq (\widetilde{\pi}^{(n)})^{-1}(\check{o})$ converge in $\mathcal{C}^{\infty}$-topology to $F^{(\infty)}$. Indeed, consider a normal coordinate chart $\eta^{(n)}: \mathscr{B}_{\check{g}^{(n)}}^B(\check{o}, r) \to \mathbb{R}^{m-s}$ of $(B,\check{g}^{(n)})$ centered at $\check{o}$, set $\mathscr{U}^{(n)} \coloneqq (\widetilde{\pi}^{(n)})^{-1}\big(\mathscr{B}_{\check{g}^{(n)}}^B(\check{o}, r)\big) \subset \mathbb{E}$ and define
\begin{equation} \label{eq:convF}
\nu^{(n)} : \mathscr{U}^{(n)} \to \mathbb{R} \,\, , \quad \nu^{(n)}(x) \coloneqq \big|\eta^{(n)}\big(\widetilde{\pi}^{(n)}(x)\big)\big|^2 \,\, .
\end{equation}
Notice that, since the manifolds $(B,\check{g}^{(n)})$ converge smoothly to $(B,\check{g}^{(\infty)})$, we can choose a uniform radius $r>0$ to define $\eta^{(n)}$ and we can assume that the maps $\nu^{(n)}$ converge smoothly to a normal coordinate chart $\eta^{(\infty)}: \mathscr{B}_{\check{g}^{(\infty)}}^B(\check{o}, r) \to \mathbb{R}^{m-s}$ of $(B,\check{g}^{(\infty)})$ centered at $\check{o}$. Hence, the maps $\nu^{(n)}$ defined in \eqref{eq:convF} converge in the $\mathcal{C}^{\infty}$-topology to the map $\nu^{(\infty)} : \mathscr{U}^{(\infty)} \to \mathbb{R}$ defined by
$$
\mathscr{U}^{(\infty)} \coloneqq (\widetilde{\pi}^{(\infty)})^{-1}\big(\mathscr{B}_{\check{g}^{(\infty)}}^B(\check{o}, r)\big) \,\, , \quad
\phi^{(\infty)}(x) \coloneqq \big|\eta^{(\infty)}\big(\widetilde{\pi}^{(\infty)}(x)\big)\big|^2 \,\, .
$$
Since $F^{(n)} = (\nu^{(n)})^{-1}(0)$, $F^{(\infty)} = (\nu^{(\infty)})^{-1}(0)$ and the fact that $(\mathbb{E},\tilde{g}^{(n)})$ converges in the pointed $\mathcal{C}^{\infty}$-topology to $(\mathbb{E}, \tilde{g}^{(\infty)})$, the Implicit Function Theorem implies that the canonical embeddings $\imath_{F^{(n)}}: F^{(n)} \to (\mathbb{E},\tilde{g}^{(n)})$ of $F^{(n)}$ converge in the $\mathcal{C}^{\infty}$-topology to the the canonical embedding $\imath_{F^{(\infty)}}: F^{(\infty)} \to (\mathbb{E}, \tilde{g}^{(\infty)})$ of $F^{(\infty)}$. Finally, since $F^{(n)}$ is flat for all $n \in \mathbb{N}$, it follows that $F^{(\infty)}$ is also flat, and this concludes the proof.
\end{proof}

\medskip
\section{Convergence to a shrinking Ricci soliton}
\label{sect:main} \setcounter{equation} 0

This section is devoted to the proof of Theorem \ref{thm:last}, from which Theorem \ref{thm:MAIN} follows immediately. We remark that the conclusion of Theorem \ref{thm:MAIN} was already established in the non-collapsed case by \cite[Theorem 4.2]{BLS19}. Accordingly, we restrict our attention to collapsed solutions.

\begin{theorem} \label{thm:last}
Let $M^m = \mathsf{G}/\mathsf{H}$ be a homogeneous space, with $\mathsf{H}$ and $\mathsf{G}$ compact, connected Lie groups. Then, every collapsed, ancient $\mathsf{G}$-homogeneous Ricci flow $g(t)$ on $M$ admits a blow-down sequence $g^{(n)}(t)$ for which the following hold: \begin{itemize}
\item[i)] there exists a compact torus $\mathsf{T} \subset N_{\mathsf{G}}(\mathsf{H})/\mathsf{H}$ acting freely on $M$ such that the metrics $g^{(n)}(-1)$ converge, in the Gromov--Hausdorff topology, to an Einstein metric $\check{g}^{(\infty)}_o$ with positive scalar curvature on the quotient $M/\mathsf{T}$;
\item[ii)] the solutions $(M,g^{(n)}(t))$ converge locally in the $\mathcal{C}^{\infty}$-topology to the product of $(M/\mathsf{T},\check{g}^{(\infty)}(t))$ and a flat factor, where $\check{g}^{(\infty)}(t)$ is the ancient Ricci flow satisfying $\check{g}^{(\infty)}(-1) = \check{g}^{(\infty)}_o$.
\end{itemize}
\end{theorem}

\begin{proof}
Let $g(t)$ be a $\mathsf{G}$-invariant, collapsed, ancient Ricci flow on $M = \mathsf{G}/\mathsf{H}$. We assume that $M = \mathsf{G}/\mathsf{H}$ is almost-effective and that it is not diffeomorphic to a torus.

Consider a sequence of positive numbers $\tau^{(n)} \to +\infty$ and define the corresponding blow-down sequence
$$
g^{(n)}(t) \coloneqq \frac{1}{\tau^{(n)}}\, g\big(\tau^{(n)}t\big), \quad t \in (-\infty, 0) \,\, .
$$
In the interest of simplifying the notation, in what follows we will use the same notation for a sequence and its subsequences. \smallskip

By \cite[Corollary 2]{BLS19}, $g(t)$ satisfies \eqref{eq:typeI} and so, up to scaling, it follows that $\big|\mathrm{sec}(g^{(n)}(-1))\big| \leq 1$ for all $n \in \mathbb{N}$. Moreover, as a result of curvature bounds \eqref{eq:typeI}, it follows from \cite[Theorem 17.2]{Ham95} that $\mathrm{diam}(M, g(t)) = O\big(\sqrt{|t|}\big)$, and so there exists $D > 0$ such that $\mathrm{diam}\big(M,g^{(n)}(-1)\big) \leq D$ for all $n \in \mathbb{N}$. Therefore, by standard Shi estimates (see \cite{Ban87}), there exists a sequence of positive numbers $\{\Gamma_k\}$ such that $g^{(n)}(-1) \in \mathscr{M}^{\mathsf{G}}\big(D,\{\Gamma_k\}\big)$ for all $n \in \mathbb{N}$. By Theorem \ref{thm:equivGH1}, the manifolds $\big(M,g^{(n)}(-1)\big)$ subconverge in the Gromov--Hausdorff topology to $\big(B = M/\mathsf{T}^s,\, h_o\big)$ as $n \to +\infty$, for some torus $\mathsf{T}^s \subset \mathsf{N}_{\mathsf{G}}(\mathsf{H})/\mathsf{H}$ acting freely on the right on $M$. Furthermore, by Theorem \ref{thm:equivGH3}, the corresponding sequence of the geometric models of $\big(M,g^{(n)}(-1)\big)$ subconverges in the pointed $\mathcal{C}^\infty$-topology to a geometric model $\big(\mathbb{E},\, \tilde{h}_o\big)$ that is locally isometric to the Riemannian product $\mathbb{R}^s \times \big(B,\, h_o\big)$.  

By Theorem \ref{thm:Lottcompactness}, the blow-down flows $\big(M, g^{(n)}(t)\big)$, corresponding to the subsequence chosen in the previous paragraph, subconverge locally in the $\mathcal{C}^{\infty}$-topology to an ancient locally homogeneous Ricci flow solution $(\mathbb{E},\, \tilde{h}(t))$. By uniqueness of the limit, it follows that $(\mathbb{E},\, \tilde{h}(-1))$ is locally isometric to $\mathbb{R}^s \times \big(B,\, h_o\big)$. Hence, the limit solution $(\mathbb{E},\, \tilde{h}(t))$ is locally isometric to the Riemannian product
$$
\mathbb{R}^s \times \big(B,\, h(t)\big) \,\, ,
$$
where $h(t)$ is the ancient homogeneous Ricci flow on the compact manifold $B$ with $h(-1) = h_o$. 

If the factor $\big(B^{m-s}, h(t)\big)$ remains collapsed, we may consider a further blow-down limit in the local $\mathcal{C}^{\infty}$-topology, denoted by $\big(Y^{m-s}, h_1(t)\big)$. Note that the class of blow-down limits is stable under further blow-downs, that is, any blow-down limit of a blow-down limit of $(M, g(t))$ is itself a blow-down limit of $(M, g(t))$. Therefore, by applying the same argument as at the beginning of the proof, we can conclude that the Riemannian product $\mathbb{R}^s \times \big(Y^{m-s}, h_1(t)\big)$ is isometric to $\mathbb{R}^{s+s'} \times \big(M/\tilde{\mathsf{T}}^{s+s'}, \check{h}_1(t)\big)$, where $\mathsf{T}^{s} \subset \tilde{\mathsf{T}}^{s+s'} \subset \mathsf{N}_{\mathsf{G}}(\mathsf{H})/\mathsf{H}$, and $\check{h}_1(t)$ is an ancient solution on the quotient $M/\tilde{\mathsf{T}}^{s+s'}$. This procedure can be iterated finitely many times, until a non-collapsed limit solution is eventually obtained on the compact factor. Since $M$ is not diffeomorphic to a torus, such compact factor has positive dimension. \smallskip

We may therefore assume that $\big(B,\, h(t)\big)$ itself is non-collapsed. Thus, by \cite[Theorem 4.2]{BLS19}, every blow-down sequence along $h(t)$ converges in the $\mathcal{C}^{\infty}$-topology to an Einstein metric $\check{g}^{(\infty)}_o$ on $B$. Since the class of blow-down limits is stable under further blow-downs, this provides a blow-down sequence along $g(t)$ that converges locally in the $\mathcal{C}^{\infty}$-topology to the product of a flat factor and the ancient Ricci flow $\check{g}^{(\infty)}(t)$ on $B$ satisfying $\check{g}^{(\infty)}(-1) = \check{g}^{(\infty)}_o$. This proves claim $ii)$.  By abuse of notation, we still denote this blow-down sequence as $g^{(n)}(t)$.

It remains to show that the manifolds $(M,g^{(n)}(-1))$ converge in the Gromov--Hausdorff topology to $(B,\check{g}^{(\infty)}_o)$ as $n \to +\infty$. Assume by contradiction that this does not hold. Then, by Theorem \ref{thm:equivGH1}, the manifolds $(M,g^{(n)}(-1))$ subconverge in the Gromov--Hausdorff topology to a Riemannian manifold $(B_1,\check{g}^{(\infty)}_1)$ which is not isometric to $(B,\check{g}^{(\infty)}_o)$, but that would contradict Theorem \ref{thm:equivGH3}.
\end{proof}

\begin{remark}
By \cite[Theorem 2.1]{BWZ04}, collapsed ancient solutions to the homogeneous Ricci flow on compact manifolds occur only on the total space of principal torus bundles. In \cite[Theorem A]{KPS25}, a Lie algebraic proof of this fact was given, together with a precise characterization of the {\it collapsing torus} which corresponds to the directions where collapse occurs along a sequence of times (see \cite[Definition 3.2]{KPS25}). We observe that, up to passing to a subsequence, the torus $\mathsf{T} \subset N_{\mathsf{G}}(\mathsf{H})/\mathsf{H}$ appearing in the proof above coincides with the collapsing torus described in \cite{KPS25}.
\end{remark}

\appendix

\medskip
\section{Locally homogeneous gradient Ricci solitons}
\label{sect:appendix} \setcounter{equation} 0

This appendix provides a proof of the following result, which extends \cite[Theorem 1.1]{PetWy09} to the setting of locally homogeneous spaces.

\begin{theorem} \label{thm:appendix}
All locally homogeneous gradient Ricci solitons are rigid.
\end{theorem}

Recall that a Riemannian manifold $(M^m, g)$ is called a {\it gradient Ricci soliton} if there exist a smooth function $f: M \to \mathbb{R}$ and a constant $\lambda \in \mathbb{R}$ such that
\begin{equation} \label{eq:soliton}
\mathrm{Ric}(g) + \mathrm{Hess}_g(f) = \lambda g \,\, .
\end{equation}
Here, $\mathrm{Hess}_g(f) \coloneqq D^g(\mathrm{d}f)$ denotes the {\it Hessian} of $f$. We also recall the following well-known formula, which does not require the metric to be complete.

\begin{lemma}
Let $(M,g)$ be a (possibly not complete) gradient Ricci soliton as in \eqref{eq:soliton}. Then, there exists a constant $c \in \mathbb{R}$ such that
\begin{equation} \label{eq:scalsol}
\mathrm{scal}(g) + |{\rm d}f|_g^2 - 2\lambda f = c \,\, .
\end{equation}
\end{lemma}

\begin{proof}
Let $\delta_g$ denote the {\it divergence} of symmetric tensor fields (see, e.g., \cite[Section 1.59]{Bes08}). Using Bochner's identity together with the soliton equation, we compute
$$
\delta_g\big(\mathrm{Hess}_g(f)\big)
= \mathrm{d}\Delta_g f + \mathrm{Hess}_g(f)(\mathrm{d}f^{\sharp},\cdot) - \lambda\,\mathrm{d}f \,\, ,
$$
where $\Delta_g$ denotes the (non-negative definite) Laplace-Beltrami operator on functions. Moreover,
\begin{equation} \label{eq:Hess(f)}
\mathrm{d}\big(|\mathrm{d}f|_g^2\big)
= 2(D^g\mathrm{d}f)(\mathrm{d}f^{\sharp})
= 2\,\mathrm{Hess}_g(f)\big(\mathrm{d}f^{\sharp},\cdot\big),
\end{equation}
and, applying the soliton equation, we obtain
$$
\Delta_g f
= -\mathrm{Tr}_g\big(\mathrm{Hess}_g(f)\big)
= -\mathrm{Tr}_g\big(-\mathrm{Ric}(g) + \lambda g\big)
= \mathrm{scal}(g) - \lambda m \,\, .
$$
Notice that $\delta_g g = 0$ and that
$$
\delta_g\big(\mathrm{Ric}(g)\big) = -\tfrac12\,\mathrm{d}\big(\mathrm{scal}(g)\big)
$$
by the second Bianchi identity. Therefore, taking the divergence of the soliton equation yields
\begin{align*}
0 &= \delta_g\big(\mathrm{Ric}(g)\big) + \delta_g\big(\mathrm{Hess}_g(f)\big) \\
&= -\tfrac12\,\mathrm{d}\big(\mathrm{scal}(g)\big) + \mathrm{d}\big(\mathrm{scal}(g)\big)
+ \tfrac12\,\mathrm{d}\big(|\mathrm{d}f|_g^2\big) - \lambda\,\mathrm{d}f \\
&= \tfrac12\,\mathrm{d}\Big(\mathrm{scal}(g) + |\mathrm{d}f|_g^2 - 2\lambda f\Big) \,\, ,
\end{align*}
which proves the claim.
\end{proof}

We now prove the following proposition, which provides a local version of \cite[Proposition 1]{PetWy09}.

\begin{proposition} \label{prop:localPW}
Let $(M,g)$ be a (possibly not complete) gradient Ricci soliton as in \eqref{eq:soliton} and assume that $\mathrm{scal}(g)$ is constant. Let $X$ be a Killing vector field on $(M,g)$ and let $V \coloneqq \big({\rm d}({\rm d}f(X))\big){}^{\sharp}$ denote the gradient of ${\rm d}f(X)$. Then, $D^gV = 0$. Moreover, if $\lambda \neq 0$, then $V = 0$ if and only if $\mathrm{d}f(X) = 0$.
\end{proposition}

\begin{proof}
The first claim follows directly from \eqref{eq:soliton} (see \cite[Proposition 1]{PetWy09}). For the second claim, we compute
$$
g(V,{\rm d}f^{\sharp}) = \mathcal{L}_{{\rm d}f^{\sharp}}\big({\rm d}f(X)\big) = {\rm Hess}_g(f)\big(\mathrm{d}f^{\sharp},X\big) \overset{\eqref{eq:Hess(f)}}{=} \tfrac12\,\mathcal{L}_X\big(|\mathrm{d}f|_g^2\big) \,\, .
$$
Moreover, by \eqref{eq:scalsol} and the fact that $\mathrm{scal}(g)$ is constant, we obtain
$$
\tfrac12\,\mathcal{L}_X\big(|\mathrm{d}f|_g^2\big) =\lambda\,\mathrm{d}f(X) \,\, ,
$$
which completes the proof.
\end{proof}

Therefore, Proposition \ref{prop:localPW} allows to apply the same argument as in \cite[Theorem 2.3]{PetWy09} to prove Theorem \ref{thm:appendix}. For the convenience of the reader, we summarize it below.

\begin{proof}[Proof of Theorem \ref{thm:appendix}]
Let $(M,g_o)$ be a (possibly not complete) locally homogeneous, gradient Ricci soliton as in \eqref{eq:soliton}. Fix a point $p \in M$, denote by $\mathfrak{g}$ the Lie algebra of {\it Killing generators} of $(M,g_o)$ at $p$ (see \cite[Section 3]{Nom60}). By \cite[Theorem 2]{Nom60}, up to restricting $M$ to a neighborhood of $p$, every element $X \in \mathfrak{g}$ can be identified with a Killing vector field defined on all of $M$.

Assume first that $\lambda = 0$, and let $g(t)$ be the Ricci flow solution with $g(0)=g_o$. In this case, $g(t)$ evolves by pullback through a smooth $1$-parameter family of diffeomorphisms. Therefore, since ${\rm scal}(g_o)$ is constant, it follows that $\partial_t \,{\rm scal}(g(t)) = 0$. By local homogeneity, this forces $\mathrm{Ric}(g(t)) = 0$, which in turn implies that $g(t)$ is flat by \cite[Theorem 3.4]{Spi93}. 

Assume now that $\lambda \neq 0$ and define
$$
\mathfrak{t} \coloneqq \big\{\big({\rm d}({\rm d}f(X))\big){}^{\sharp} : X \in \mathfrak{g} \} \,\, .
$$
By Proposition \ref{prop:localPW}, $\mathfrak{t}$ is a Lie algebra of parallel vector fields on $(M,g_o)$. We denote by $\mathcal{V} \subset TM$ the distribution induced by $\mathfrak{t}$ and by $\mathcal{H}$ its $g_o$-orthogonal distribution. By construction, both $\mathcal{V}$ and $\mathcal{H}$ are parallel, and hence integrable. Moreover, the integral submanifolds tangent to $\mathcal{V}$ are flat. Therefore, up to restricting $M$ to a neighborhood of $p$, it follows that $(M,g_o)$ is isometric to the product $\mathbb{R}^s \times (B,\check{g})$, where $(B,\check{g})$ denotes the integral submanifold of $\mathcal{H}$ through $p$, endowed with the induced metric. By \cite[Lemma 2.1]{PetWy09}, the potential function $f$ splits as $f=(f_1,f_2)$, and
$$
\mathrm{Ric}(\check{g}) + \mathrm{Hess}_{\check{g}}(f_2) = \lambda \check{g} \,\, .
$$
Up to iterating this procedure finitely many times, we may assume that every Killing vector field $X'$ of $(B,\check{g})$ satisfies
$$
{\rm d}({\rm d}f_2(X')) = 0 \,\, ,
$$
which in turns implies that ${\rm d}f_2(X') = 0$ by Proposition \ref{prop:localPW}. Hence, every Killing vector field of $(B,\check{g})$ leaves $f_2$ invariant, and since $(B,\check{g})$ is locally homogenous, this forces $f_2$ to be constant. Consequently, $(B,\check{g})$ is Einstein.
\end{proof}

\medskip
\section{Symmetrization along the collapsed directions}
\label{sect:appendix2} \setcounter{equation} 0

For the sake of notation, we introduce the following definitions.

\begin{definition}
Let $\mathsf{G}$ be a compact Lie group and let $\mathsf{T}$ be a torus. A {\it $\mathsf{G}$-equivariant Riemannian principal $\mathsf{T}$-bundle} consists of the following data: \begin{itemize}
\item[$\bcdot$] a principal $\mathsf{T}$-bundle $\pi: M \to B = M/\mathsf{T}$ with compact base;
\item[$\bcdot$] an almost-effective left action of $\mathsf{G}$ on $M$ commuting with the right action of $\mathsf{T}$;
\item[$\bcdot$] a $\mathsf{G}$-invariant, though not necessarily $\mathsf{T}$-invariant, Riemannian metric $g$ on $M$;
\item[$\bcdot$] a metric $\check{g}$ on $B$ that is invariant under the left action of $\mathsf{G}$ induced by the projection $\pi$.
\end{itemize}
\end{definition}

\begin{definition} \label{def:coll-bound}
Let $\{\Gamma_k\}$ be a sequence of positive numbers, and let $\iota > 0$ and $0 < \epsilon \ll \iota$ be positive constants. A $\mathsf{G}$-equivariant Riemannian principal $\mathsf{T}$-bundle $\pi : (M,g) \to (B,\check{g})$ is said to be: \begin{itemize}
\item[$i)$] {\it $\epsilon$-collapsed} if the projection $\pi: (M, g) \to (B, \check{g})$ is an $\epsilon$-Gromov--Hausdorff approximation and an $\epsilon$-almost Riemannian submersion in the $\mathcal{C}^2$-sense;
\item[$ii)$] {\it $\big(\iota,\{\Gamma_k\}\big)$-bounded} if both $(M, g)$ and $(B, \check{g})$ have $\{\Gamma_k\}$-bounded geometry, the restriction of $g$ to each $\mathsf{T}$-orbit in $M$ is flat, the projection $\pi: (M, g) \to (B, \check{g})$ is $\{\Gamma_k\}$-bounded, and $\mathrm{inj}(B, \check{g}) \geq \iota$.
\end{itemize}
\end{definition}

Compact homogeneous spaces with vanishing Euler characteristic provide natural examples of  $\mathsf{G}$-equivariant Riemannian principal $\mathsf{T}$-bundles with flat fibers (see Proposition \ref{prop:flatorb}). The following is a special case of the more general definitions given in \cite{NaTian18}.

\smallskip

Given a principal $\mathsf{T}$-bundle $\pi: M \to B$, we call {\it trivializing chart for $\pi : M \to B$} every quadruple $(\mathscr{U},\Lambda,\psi,\eta)$ given by: \begin{itemize}
\item[$\bcdot$] an open set $\mathscr{U} \subset B$;
\item[$\bcdot$] a local trivialization $\psi : \pi^{-1}(\mathscr{U}) \to \mathscr{U} \times \mathsf{T}$ for $\pi: M \to B$;
\item[$\bcdot$] a lattice $\Lambda \subset \mathbb{R}^s$, called {\it covering group of $(\mathscr{U},\Lambda,\psi,\eta)$};
\item[$\bcdot$] a coordinate chart $\eta: \mathscr{U} \to \eta(\mathscr{U}) \subset \mathbb{R}^{m-s}$.
\end{itemize}
Given a trivializing chart $(\mathscr{U},\Lambda,\psi,\eta)$ for $\pi : M \to B$, we call {\it weak coordinate patch associated to $(\mathscr{U},\Lambda,\psi,\eta)$} the local diffeomorphism
\begin{equation} \label{eq:weakcoord}
\xi : \eta(\mathscr{U}) \times \mathbb{R}^s \to M \,\, , \quad \xi \coloneqq \psi^{-1} \circ (\mathrm{Id}_{\mathscr{U}} \times \varpi_{\Lambda}) \circ (\eta^{-1} \times \mathrm{Id}_{\mathbb{R}^s}) \,\, ,
\end{equation}
where $\varpi_{\Lambda} : \mathbb{R}^s \to \mathbb{R}^s/\Lambda \simeq \mathsf{T}$ denotes the Lie group covering map associated to $\Lambda$.

A {\it trivializing atlas for $\pi : M \to B$} is a collection $\mathcal{A} = \{(\mathscr{U}_{\alpha},\Lambda,\psi_{\alpha},\eta_{\alpha})\}$, where each element $(\mathscr{U}_{\alpha},\Lambda,\psi_{\alpha},\eta_{\alpha})$ is a trivializing chart for $\pi : M \to B$ with the same covering group $\Lambda$. In this case, we say $\Lambda$ is the {\it covering group of $\mathcal{A}$} and that $\mathcal{A}$ is {\it subordinate to the open covering $\{\mathscr{U}_{\alpha}\}$} of $B$.

\begin{definition}[c.f.\ \cite{NaTian18}, Definition 2.8 and Definition A.3]
Let $\pi: (M,g) \to (B,\check{g})$ be a $\mathsf{G}$-equivariant Riemannian principal $\mathsf{T}$-bundle. Let $r>0$ be a radius and let $\{L_k\}$ be a sequence of positive numbers. A trivializing chart $(\mathscr{U},\Lambda,\psi,\eta)$ for $\pi: (M,g) \to (B,\check{g})$ is called {\it $\big(r,\{L_k\}\big)$-bounded} if the following conditions hold.
\begin{itemize}
\item[$i)$] $\mathscr{U} = \mathscr{B}_{\check{g}}^B(b, r)$, for some $b \in B$, and $\eta$ is a normal coordinate chart of $(B,\check{g})$ centered at $b$.
\item[$ii)$] The pullback of $g$ via the weak coordinate patch $\xi : \mathscr{B}^{\mathbb{R}^{m-s}}(0, r) \times \mathbb{R}^s \to M$ associated to $(\mathscr{U},\Lambda,\psi,\eta)$ satisfies the following properties:
\begin{gather}
e^{-\frac{L_0}{10}}\delta_{ij} \leq (\xi^*g)_{ij}(x,y) \leq e^{\frac{L_0}{10}}\delta_{ij} \,\, \text{as bilinear forms for all $(x,y) \in \mathscr{B}^{\mathbb{R}^{m-s}}(0, r) \times \mathbb{R}^s$} \,\, , \label{eq:rLbounded1} \\
r^k\cdot\big\|(\xi^*g)_{ij}\big\|_{\mathcal{C}^k} \leq L_k \quad \text{ for all $1 \leq i, j \leq m$ and $k \in \mathbb{N}$} \,\, . \label{eq:rLbounded2}
\end{gather}
\end{itemize}
A trivializing atlas $\mathcal{A} = \{(\mathscr{U}_{\alpha},\Lambda,\psi_{\alpha},\eta_{\alpha})\}$ for $\pi: (M,g) \to (B,\check{g})$ is called {\it $\big(r,\{L_k\}\big)$-bounded} if each trivializing chart $(\mathscr{U}_{\alpha},\Lambda,\psi_{\alpha},\eta_{\alpha})$ of $\mathcal{A}$ is $\big(r,\{L_k\}\big)$-bounded.
\end{definition}

\begin{remark} \label{rmk:diagram}
Let $\pi : (M,g) \to (B,\check{g})$ be a $\mathsf{G}$-equivariant Riemannian principal $\mathsf{T}$-bundle, fix a $\big(r,\{L_k\}\big)$-bounded trivializing chart $(\mathscr{U},\Lambda,\psi,\eta)$ and denote by $\xi$ the associated weak coordinate patch. Let also $X \to \pi^{-1}(\mathscr{U})$ be the universal covering map of $\pi^{-1}(\mathscr{U})$. Then, $\psi$ lifts to a diffeomorphism $\tilde{\psi} : X \to \mathscr{U} \times \mathbb{R}^s$, which intertwines the deck action of $\pi_1\big(\pi^{-1}(\mathscr{U})\big)$ on $X$ with the right action of $\Lambda$ on $\mathscr{U} \times \mathbb{R}^s$. More precisely, there exists a group isomorphism $\jmath : \pi_1\big(\pi^{-1}(\mathscr{U})\big) \to \Lambda$ such that, for every $x \in X$ and $\gamma \in \pi_1\big(\pi^{-1}(\mathscr{U})\big)$, if $\tilde{\psi}(x) = (b,y) \in \mathscr{U} \times \mathbb{R}^s$, then $\tilde{\psi}(x \cdot \gamma) = (b,\, y + \jmath(\gamma))$. Since the following diagram is commutative, it follows that the pulled back metric $\xi^*g$ is $\Lambda$-invariant.

$$\begin{tikzpicture}[>=Latex,node/.style={}]
\node[node] (A) at (0,4) {$X$};
\node[node] (B) at (4,4) {$\mathscr{U} \times \mathbb{R}^s$};
\node[node] (C) at (9,4) {$\mathscr{B}^{\mathbb{R}^{m-s}}(0, r) \times \mathbb{R}^s$};
\node[node] (D') at (-1,2) {$M \supset$};
\node[node] (D) at (0,2) {$\pi^{-1}(\mathscr{U})$};
\node[node] (E) at (4,2) {$\mathscr{U} \times \mathsf{T}$};
\node[node] (F') at (-0.5,0) {$B \supset$};
\node[node] (F) at (0,0) {$\mathscr{U}$};
\draw[->] (A) -- node[above]{$\tilde{\psi}$} (B);
\draw[->] (B) -- node[above]{$\eta \times \mathrm{Id}_{\mathbb{R}^s}$} (C);
\draw[->] (D) -- node[above, pos=0.55]{$\psi$} (E);
\draw[->] (A) -- node[left]{} (D);
\begin{scope}[on background layer]
\draw[->] (B) -- node[left, pos=0.3]{$\mathrm{Id}_{\mathscr{U}} \times \varpi_{\Lambda}$} (E);
\end{scope}
\draw[->] (D) -- node[left]{$\pi$} (F);
\draw[->,thick,preaction={draw=white, line width=4pt}] (C) -- node[above, pos=0.4]{$\xi$} (D);
\draw[->] (E) -- node[above]{${\rm pr}_1$} (F);
\end{tikzpicture}$$
\end{remark}

The following result, which is a special case of the more general theorem stated in \cite{NaTian18}, establishes the existence of $\big(r,\{L_k\}\big)$-bounded trivializing atlases for $\mathsf{G}$-equivariant Riemannian principal $\mathsf{T}$-bundles, under suitable geometric assumptions.

\begin{theorem}[c.f.\ \cite{NaTian18}, Theorem A.2] \label{thm:weakcoord}
Let $\{\Gamma_k\}$ be a sequence of positive numbers, and let $\iota > 0$ and $0 < \epsilon \ll \iota$ be positive constants. Let $\pi: (M,g) \to (B,\check{g})$ be an $\epsilon$-collapsed, $\big(\iota,\{\Gamma_k\}\big)$-bounded $\mathsf{G}$-equivariant principal $\mathsf{T}$-bundle. Then there exist $r > 0$ and a sequence of positive numbers $\{L_k\}$, both depending only on the data $(m, s, \{\Gamma_k\}, \iota, \mathsf{G})$ and independent of $\epsilon$, such that for every open cover $\{\mathscr{U}_i = \mathscr{B}_{\check{g}}^B(b_i, r)\}$ of $B$, there exists a $\big(r,\{L_k\}\big)$-bounded trivializing atlas for $\pi: (M,g) \to (B,\check{g})$ subordinate to the cover $\{\mathscr{U}_i\}$.
\end{theorem}

By means of Theorem \ref{thm:weakcoord}, it is possible to prove the following estimate, from which Theorem \ref{thm:equivGH2} follows immediately.

\begin{theorem} \label{thm:g-fg}
Let $\{\Gamma_k\}$ be a sequence of positive numbers, and let $\iota > 0$ and $0 < \epsilon \ll \iota$ be positive constants. Let $\pi: (M,g) \to (B,\check{g})$ be an $\epsilon$-collapsed, $\big(\iota,\{\Gamma_k\}\big)$-bounded $\mathsf{G}$-equivariant principal $\mathsf{T}$-bundle. Then there exists a sequence of positive numbers $\{C_k\}$, depending only on the data $(m, s, \{\Gamma_k\}, \iota, \mathsf{G})$ and independent of $\epsilon$, such that
$$
\big\|g -f^*g\big\|_{\mathcal{C}^k} \leq C_k\, \epsilon \quad \text{for all $f \in \mathsf{T}$ and $k \geq 0$} \,\, .
$$
\end{theorem}

\begin{proof}
By Theorem \ref{thm:weakcoord}, there exists a $\big(r,\{L_k\}\big)$-bounded trivializing atlas $\mathcal{A}$ for $\pi: (M,g) \to (B,\check{g})$. Since $M$ is compact, it is sufficient to prove the result on a trivializing chart $(\mathscr{U},\Lambda,\psi,\eta) \in \mathcal{A}$. Let $\{e_1,{\dots},e_s\}$ be a set of generators for the covering group $\Lambda$ and denote by $\Delta \coloneqq \{v = v^ie_i : 0 \leq v^i <1 \}$ the corresponding fundamental domain in $\mathbb{R}^s$. Let $\xi : \mathscr{B}^{\mathbb{R}^{m-s}}(0, r) \times \mathbb{R}^s \to M$ be the weak coordinate patch associated to $(\mathscr{U},\Lambda,\psi,\eta)$ and set $\tilde{g} \coloneqq \xi^*g$ for the sake of notation.

Since $\pi: (M,g) \to (B,\check{g})$ is $\epsilon$-collapsed (see Definition \ref{def:coll-bound}) and the action of $\Lambda$ is equivalent to the deck action of $\pi_1\big(\pi^{-1}(\mathscr{U})\big)$ (see Remark \ref{rmk:diagram}), the orbits of $\Lambda$ are $\epsilon$-dense in each fibre $\{x\} \times \mathbb{R}^s \subset \mathscr{B}^{\mathbb{R}^{m-s}}(0, r) \times \mathbb{R}^s$ with respect to $\tilde{g}$. More precisely, for every $x \in \mathscr{B}^{\mathbb{R}^{m-s}}(0, r)$ and $y_1, y_2 \in \mathbb{R}^s$, there exists $\lambda \in \Lambda$ such that
$\mathtt{d}_{\tilde{g}}\big((x,y_2),(x,y_1+\lambda)\big) < \epsilon$. By \eqref{eq:rLbounded1}, it follows that for any $y \in \mathbb{R}^s$ there exists $\lambda \in \Lambda$ such that $|y -\lambda | < e^{\frac{L_0}{20}} \epsilon$, which in turn implies that 
\begin{equation} \label{eq:lengthLambda}
|v| < 10s\, e^{\frac{L_0}{20}} \epsilon \quad \text{for all $v \in \Delta$} \,\, .
\end{equation}

Fix an element $f \in \mathsf{T}$ and notice that it preserves the open set $\pi^{-1}(\mathscr{U})$ and lifts to a translation
\begin{equation} \label{eq:liftf}
\tilde{f}: \mathscr{B}^{\mathbb{R}^{m-s}}(0, r) \times \mathbb{R}^s \to \mathscr{B}^{\mathbb{R}^{m-s}}(0, r) \times \mathbb{R}^s \,\, , \quad \tilde{f}(x,y) = (x,y+\tau) \,\, ,
\end{equation}
where $\tau \in \Delta$ is the unique element such that $\varpi_{\Lambda}(\tau) = f$. By \eqref{eq:liftf}, it follows that the components of the metric $\tilde{f}^*\tilde{g}$ are given by
$$
\big(\tilde{f}^*\tilde{g}\big)_{ij}(x,y) = \tilde{g}_{ij}(x,y+\tau) \,\, .
$$
Therefore, by the Mean Value Theorem, we get
\begin{equation} \label{eq:mvt}
\big\| \tilde{g}_{ij} -\big(\tilde{f}^*\tilde{g}\big)_{ij} \big\|_{\mathcal{C}^k} \leq \tilde{C}(m,k) \big\|\tilde{g}_{ij}\big\|_{\mathcal{C}^{k+1}} \, |\tau|_{\mathrm{st}} \,\, .
\end{equation}
for some constant $\tilde{C}(m,k)>0$. Finally, the thesis follows from \eqref{eq:rLbounded2}, \eqref{eq:lengthLambda} and \eqref{eq:mvt}.
\end{proof}

\end{document}